\documentclass[12pt,reqno]{amsart}
\usepackage{amsmath,amssymb,graphicx,mathrsfs,amsopn,stmaryrd,amsthm,color,bm,extarrows}
\usepackage[dvipsnames]{xcolor}
\usepackage[colorlinks=true,citecolor=blue,linkcolor=blue
]{hyperref}
\usepackage[final]{showlabels}
\usepackage[english]{babel}
\usepackage[mathscr]{euscript}
\usepackage{geometry}
\usepackage{collectbox}
\makeatletter
\newcommand{\mybox}{%
    \collectbox{%
        \setlength{\fboxsep}{1pt}%
        \fbox{\BOXCONTENT}%
    }%
}
\makeatother

\let\emptyset\varnothing
\usepackage{dsfont}
\usepackage{tikz}    
\usetikzlibrary{arrows,matrix,decorations.markings,shapes.geometric}   
\usetikzlibrary{decorations.pathreplacing}

\usepackage{xy}
\allowdisplaybreaks
\xyoption{all}
\numberwithin{equation}{section}
\newtheorem{thm}{Theorem}[section]
\newtheorem{prop}[thm]{Proposition}
\newtheorem{lem}[thm]{Lemma}
\newtheorem{cor}[thm]{Corollary}

\theoremstyle{definition} 

\newtheorem{dfn}[thm]{Definition}
\theoremstyle{remark}
\newtheorem{rem}[thm]{Remark}

\newcommand{\beq}{\begin{equation}}
\newcommand{\eeq}{\end{equation}}
\newcommand{\be}{\begin{equation*}}
\newcommand{\ee}{\end{equation*}}

\newcommand{\bC}{\mathbb{C}}

\newcommand{\bZ}{\mathbb{Z}}

\newcommand{\bF}{\mathbb{F}}
\newcommand{\bS}{\mathbb{S}}

\newcommand{\bN}{\mathbb{N}}

\newcommand{\ttb}{\mathtt{b}}

\newcommand{\mfk}{\mathfrak}
\newcommand{\g}{\mathfrak{g}}
\newcommand{\gl}{\mathfrak{gl}}

\newcommand{\fksp}{\mathfrak{sp}}

\newcommand{\fksl}{\mathfrak{sl}}

\newcommand{\rY}{\mathrm{Y}}

\newcommand{\End}{\mathrm{End}}
   
\newcommand{\sdet}{{\mathrm{sdet}}}   
\newcommand{\gr}{{\mathrm{gr}}}

\newcommand{\tl}{\tilde}
\newcommand{\wtl}{\widetilde}
\newcommand{\gge}{\geqslant}
\newcommand{\lle}{\leqslant}

\newcommand{\La}{\Lambda}

\newcommand{\Y}{{\mathscr{Y}}}
\newcommand{\X}{{\mathscr{X}}}

\newcommand{\SY}{{\mathscr{SY}}}

\newcommand{\YDrN}{{\mathscr{Y}^{\mathrm{Dr}}_N}}
\newcommand{\SX}{{\mathscr{SX}}}

\begin{document}
\pagestyle{myheadings}
\setcounter{page}{1}

\title[A Drinfeld type presentation of twisted Yangians]{A Drinfeld type presentation of twisted Yangians}

\author{Kang Lu}
\author{Weiqiang Wang}
\address{Department of Mathematics, University of Virginia, 
Charlottesville, VA 22903, USA}\email{kang.lu@virginia.edu, ww9c@virginia.edu}

\author{Weinan Zhang}
\address{Department of Mathematics and New Cornerstone
Science Laboratory, The University of Hong Kong, Hong Kong SAR, P.R.China}
\email{mathzwn@hku.hk}

\subjclass[2020]{Primary 17B37.}
	\keywords{Drinfeld presentation, twisted Yangians}

	\begin{abstract}
		We develop a Gauss decomposition approach to establish a Drinfeld type current presentation for Olshanski's twisted Yangians associated to the orthogonal Lie algebras (also called twisted Yangians of type AI), settling a longstanding open problem. We expect that this will open the door for finding current presentations for other twisted Yangians. 
	\end{abstract}
	
	\maketitle
 \setcounter{tocdepth}{1}
	\tableofcontents

\thispagestyle{empty}

\section{Introduction}

\subsection{The problem}
Motivated by the quantum inverse scattering method \cite{FST79,FT84}, Drinfeld introduced Yangians which have provided a powerful mathematical framework for understanding the symmetries and algebraic structures underlying various physical phenomena; see \cite{Lo16} and the references therein. 
Yangians are deformations of current algebras, and they admit a remarkable current presentation, also known as Drinfeld's new presentation \cite{Dr88}. Olshanski's twisted Yangians \cite{Ol92} were defined via an R-matrix presentation. They are closely related to classical Lie algebras of type BCD and their representations \cite{MNO96}, just as the Yangian for the general linear Lie algebra $\gl_N$ is intimately related to representations of $\gl_N$; these rich topics are the subjects of Molev's book \cite{Mol07}. 

Twisted Yangians are deformations of twisted current algebras, and finding a Drinfeld type current presentation for them has been a long-standing problem since \cite{Ol92}. The current presentations of twisted Yangians have potential applications in the study of quantum integrable systems with boundary conditions \cite{Skl88}, such as obtaining recurrence relations and action relations, scalar products and norms of Bethe vectors for open spin chain models \cite{BKI93,Sl07,Reg23}. The current presentation also reveals a maximal commutative subalgebra, which is useful in representation theory, and allows the construction of $q$-characters for twisted Yangians, similar to the case of usual Yangians \cite{K95}. We refer to \S\ref{subsec:future} below for discussion of further potential applications.

\subsection{Our approach}

The goal of this paper is to establish a Drinfeld type current presentation for the twisted Yangians $\Y_N$ associated to the orthogonal Lie algebras $\mathfrak{so}_N$ (which can be referred to as twisted Yangians of type AI, as we explain below). 
We shall develop a Gauss decomposition approach to produce the desired current generators and relations, despite the absence of a natural triangular decomposition of $\Y_N$. 

Gauss decomposition is a conceptual approach for establishing the isomorphism between R-matrix presentations and current presentations of Yangians; see Brundan-Kleshchev \cite[\S 5]{BK05} for $\rY(\gl_N)$. 
Establishing such an isomorphism for Yangians of type BCD is technically  more challenging and was accomplished more recently by Jing-Liu-Molev \cite{JLM18}. In classifying the finite-dimensional irreducible representations of $\Y_N$ \cite{Mol98}, it is necessary to separate the cases based on the parity of $N$ (corresponding to type B and D Lie algebras). Similar type-preserving embeddings were used crucially in establishing the isomorphisms between the R-matrix presentation and the current presentation of Drinfeld's Yangians of type BCD \cite{JLM18}.  

In our approach, the computations of the Gauss decomposition of $\Y_N$ for small $N$ play a fundamental role in the general case. The case for $N=2$ produces the rank one generators and relations, including a novel relation  (typical for twisted Yangians) which identifies the upper and lower triangular currents. The next case which is not obvious but crucial to us turns out to be $N=3$ (not $N=4$); this computation produces a new Serre relation among current generators, which is much more sophisticated to formulate and prove than its counterpart for Yangians. Note that Brown \cite{Br16} gave a {\em parabolic} version of Drinfeld presentation for $\Y_3$, in which Serre relations do not appear. 

One conceptual way to classify irreducible symmetric pairs is to use  Satake diagrams (which roughly speaking are Dynkin diagrams with decorations); cf. \cite{OV90}. 
In this paper it is crucial for us to take the point of view that the twisted Yangian $\Y_N$ is associated with a Satake diagram of split type AI (instead of regarding that $\Y_N$ is associated with $\mathfrak{so}_N$). From this view, it is natural for us to consider embeddings $\Y_{M} \rightarrow \Y_{N}$ for $M\lle N$ where $M, N$ are not assumed to have the same parity. Via the detailed study of such embeddings, we show that the rank 1 and rank 2 generators and relations established for $\Y_2$ and $\Y_3$ give rise to all the generators and defining relations for the current presentation of $\Y_N$, for all $N$. We remark that the current presentation of $\Y_N$ manifestly uses the type $A_{N-1}$ Cartan matrix (or the corresponding split Satake diagram of type AI) as the input; see Theorem~\ref{mainthm}. 

A variant of the twisted Yangians $\Y_N$, known as special twisted Yangians and denoted by $\SY_N$ here, is also studied systematically in this paper. We obtain a current presentation of $\SY_N$ (see Theorem \ref{mainthm2}). The current presentations of $\Y_N$ and $\SY_N$ show both striking similarities and differences with the Drinfeld presentations for the Yangians $\rY(\mfk{gl}_N)$ and $\rY(\mfk{sl}_N)$.

\subsection{Future directions}
\label{subsec:future}

Motivated by works in the math physics literature \cite{BR01, Ra01}, Brundan-Kleshchev \cite{BK05} established various parabolic presentations of $\rY(\gl_N)$, which interpolate between the R-matrix presentation and Drinfeld presentation. These constructions were instrumental in their sequels in developing direct connections between representation theory of shifted Yangians and finite W-algebras of type A. The current presentation of $\Y_N$ in this paper will lead to shifted twisted Yangians and applications to finite W-algebras of classical type; see Brown \cite{Br16} (and the references therein) for earlier attempts in special cases. Such connections have been pursued in \cite{LPT+25}. 

Yangians and affine quantum groups are Hopf algebra deformations of current algebras and loop algebras, respectively. On the other hand, twisted Yangians and affine $\imath$quantum groups are coideal subalgebras of Yangians and affine quantum groups, and they are coideal deformations of twisted current algebras and twisted loop algebras associated to symmetric pairs (or Satake diagrams). The $\imath$quantum groups arise from quantum symmetric pairs introduced by G.~Letzter and generalized by Kolb (cf. \cite{LW21} and the references therein). Twisted Yangians associated to symmetric pairs in a $\mathcal J$-presentation were introduced in \cite{Ma02} and in R-matrix presentations (in classical type) were introduced in \cite{Ol92,MR02, GR16}; see also Guay-Regelskis-Wendlandt \cite{GRW17}. From the viewpoint of symmetric pairs and Satake diagrams, $\Y_N$ is the twisted Yangian of type AI (as it is associated to a symmetric pair of type AI). 

Yangians can also be obtained via a degeneration from affine quantum groups. In a companion paper \cite{LWZ24}, the authors have developed a degeneration approach to produce algebras with current presentations from affine $\imath$quantum groups in Drinfeld type presentations (by Ming Lu and two of the authors \cite{LW21, Z22}); these algebras are expected to form a wider class of twisted Yangians associated to split Satake diagrams in current presentations. This is supported by the case of type AI, as the resulting current presentations from the degeneration and from the Gauss decomposition in this paper match perfectly. 

The constructions in this paper will likely pave the way for finding Drinfeld type current presentations for general twisted Yangians; see \S \ref{subsec:gen} for further discussions. As all twisted Yangians of split type ADE share the same rank 1 and 2 Satake subdiagrams (of split $A_1$ and $A_2$ types, respectively), we expect that the rank 1 and rank 2 relations found in this paper give rise to a current presentation for twisted Yangians of {\em split} type D and E (associated to the corresponding Cartan matrices); note that twisted Yangians of type D but not of type E were defined in R-matrix form by Guay-Regelskis \cite{GR16}. 

In contrast to the setting of Yangians where there is a unique rank 1 (associated to $\fksl_2$) and a uniform presentation for rank 2, twisted Yangians exhibit several real rank 1 and many more rank 2 cases; here the (real) rank is understood in the sense of symmetric pairs and Satake diagrams. Accordingly the current presentations of these rank 1 and 2 twisted Yangians must be established separately. Moreover, there is also a real rank 0 and compact rank 1 twisted Yangian of type AII (i.e., Olshanski's twisted Yangian associated to $\fksp_2$) that requires   separate treatment. These rank 0-1-2 generators and relations are then expected to give rise to current presentations of general twisted Yangians.

\subsection{Organization}

The paper is organized as follows. In the preliminary Section \ref{sec:tY}, we review twisted Yangians, extended twisted Yangians and special twisted Yangians. Some basic symmetries of these algebras are recalled. 
In Section \ref{sec:GD}, we study the new current generators arising from the Gauss decomposition for $\Y_N$ and establish some simple relations among them. We also describe the behaviors of the current generators under various symmetries and embeddings of twisted Yangians. 

Various relations in the (extended) twisted Yangians for $N=2$ and $N=3$ are computed in both generating function form and component-wise form in Section \ref{sec:low rank}. In Section \ref{sec:Drinfeld}, we formulate and prove the Drinfeld type presentations for twisted Yangian $\Y_N$ and special twisted Yangian $\SY_N$.

\vspace{2mm}
\noindent {\bf Acknowledgement.}
The authors are partially supported by the NSF grants DMS-2001351 and DMS-2401351. WZ is partially supported by the New Cornerstone Foundation through the New Cornerstone Investigator grant awarded to Xuhua He.

\section{Twisted Yangians of type AI}
\label{sec:tY}

In this section, we recall the basics about Yangians and twisted Yangians from \cite{MNO96,Mol07}. We observe that the twisted Yangians can be defined with a simpler symmetry relation. 

\subsection{Yangians}
\label{subsec:Y}

The  \textit{Yangian} $\rY(\gl_N)$ corresponding to the Lie algebra $\gl_N$ is a unital associative algebra with generators $t_{ij}^{(r)}$, where $1\lle i,j\lle N$ and $r\in\bZ_{>0}$, and the defining relations 
\[
(u-v)[t_{ij}(u),t_{kl}(v)]=t_{kj}(u)t_{il}(v)-t_{kj}(v)t_{il}(u).
\]
Here we have used the generating series in an indeterminate $u$
\[
t_{ij}(u)=  \delta_{ij}+t_{ij}^{(1)}u^{-1}+t_{ij}^{(2)}u^{-2}+\cdots,
\]
where $\delta_{ij}$ denotes the Kronecker delta.

Equivalently, the Yangian $\rY(\gl_N)$ has an R-matrix presentation as follows. Let $R(u)$ be the Yang R-matrix
\begin{align} \label{Ru}
R(u)=1-\frac{P}{u}\in \End(\bC^N\otimes \bC^N)[u^{-1}],\quad \text{where} \quad P=\sum_{i,j=1}^N E_{ij}\otimes E_{ji},
\end{align}
and
\[
T(u)=\sum_{i,j=1}^N E_{ij}\otimes t_{ij}(u)\in \End(\bC^N)\otimes\rY(\gl_N)[[u^{-1}]]. 
\]
Then the defining relations of $\rY(\gl_N)$ can be written as (cf. \cite[Chap. 1]{Mol07})
\be
R(u-v)T_1(u)T_2(v)=T_2(v)T_1(u)R(u-v)\in  \End(\bC^N\otimes \bC^N)\otimes\rY(\gl_N)[[u^{-1}]].
\ee

Note that the Yang R-matrix satisfies the Yang-Baxter equation
\be
R_{12}(u-v)R_{13}(u)R_{23}(v)=R_{23}(v)R_{13}(u)R_{12}(u-v).
\ee

Let $g(u)$ be any formal power series in $u^{-1}$ with leading term $1$,
\[
g(u)=1+g_1u^{-1}+g_2u^{-2}+\cdots\in \bC[[u^{-1}]].
\]
There is an automorphism of $\rY(\gl_N)$ defined by
\beq\label{eq:mu_f-A}
\mu_{g(u)}:T(u)\to g(u)T(u).
\eeq

The Yangian for $\mathfrak{sl}_N$, denoted by $\rY(\mathfrak{sl}_N)$, is the subalgebra of $\rY(\gl_N)$ which consists of all elements stable under all the automorphisms of the form \eqref{eq:mu_f-A}.

Consider the filtration on $\rY(\gl_N)$ obtained by setting
\begin{align}  \label{filter:Y}
\deg t_{ij}^{(r)}=r-1
\end{align}
for every $r\gge 1$. Denote by $\mathrm{gr}\rY(\gl_N)$ the associated graded algebra. We write $\bar t_{ij}^{(r)}$ the image of $t_{ij}^{(r)}$ in $\mathrm{gr}\rY(\gl_N)$. Then the map
\[
\mathrm{U}(\gl_N[z])\to \mathrm{gr}\rY(\gl_N), \qquad e_{ij}\otimes z^{r}\mapsto \bar t_{ij}^{(r+1)},
\]
induces an Hopf algebra isomorphism.  Let $\mathcal F_s(\rY(\gl_N))$ be the subspace of $\rY(\gl_N)$ spanned by elements of degree $\lle s$. 

\subsection{Twisted Yangians}

For $1\lle i,j\lle N$ and an indeterminate $u$, we introduce the generating series 
\beq\label{siju}
s_{ij}(u)=  \delta_{ij}+s_{ij}^{(1)}u^{-1}+s_{ij}^{(2)}u^{-2}+\cdots.
\eeq

\begin{dfn}\label{tydef}
The \textit{extended twisted Yangian} $\X_N$ corresponding to the Lie algebra $\mathfrak o_N$ is the unital associative algebra  with generators $s_{ij}^{(r)}$, where $1\lle i,j\lle N$ and $r\in\bZ_{>0}$, whose generating series \eqref{siju} satisfies the following quaternary relation:
\beq\label{quater u}
\begin{split}
(u^2-v^2)[s_{ij}(u),s_{kl}(v)]=&\, (u+v)(s_{kj}(u)s_{il}(v)-s_{kj}(v)s_{il}(u))\\
- &\, (u-v)(s_{ik}(u)s_{jl}(v)-s_{ki}(v)s_{lj}(u))\\
&\, \qquad \ \  +  s_{ki}(u)s_{jl}(v)-s_{ki}(v)s_{jl}(u).
\end{split}
\eeq

The  \textit{twisted Yangian} $\Y_N$ corresponding to the Lie algebra $\mathfrak o_N$ is a unital associative algebra with the same generators as for $\X_N$ above, whose generating series satisfy the quaternary relation \eqref{quater u}
and an additional symmetry relation
\begin{align} \label{s11}
    s_{11}(u)=s_{11}(-u).
\end{align} 
\end{dfn}
If follows by definitions that we have a canonical epimorphism $\X_N \twoheadrightarrow \Y_N$. 

\begin{lem} \label{lem:s11}
Suppose $s_{ij}(u)$, $1\lle i,j\lle N$, satisfy the quaternary relation \eqref{quater u}. Then the relation \eqref{s11} is equivalent to the following symmetry relations
\beq\label{sym u}
s_{ji}(-u)=  s_{ij}(u)+\frac{s_{ij}(u)-s_{ij}(-u)}{2u}
\eeq
for all $1\lle i,j\lle N$.
\end{lem}

\begin{proof}
Setting $j=i$ in \eqref{sym u}, we obtain that $s_{ii}(u)=s_{ii}(-u)$ for all $1\lle i\lle N$. 

On the other hand, we assume \eqref{quater u} and \eqref{s11}. Setting $u=-v$ and $i=k=1$ in \eqref{quater u}, one obtains that
\[
2v(s_{11}(-v)s_{jl}(v)-s_{11}(v)s_{lj}(-v))+s_{11}(-v)s_{jl}(v)-s_{11}(v)s_{jl}(-v)=0.
\]
Since $s_{11}(v)$ is invertible and $s_{11}(v)=s_{11}(-v)$ by \eqref{s11}, the above relation is simplified to be
\[
2v(s_{jl}(v)-s_{lj}(-v))+s_{jl}(v)-s_{jl}(-v)=0,
\]
which is exactly the symmetry relations \eqref{sym u}.
\end{proof}

\begin{rem}
Thanks to Lemma \ref{lem:s11}, our definition of twisted Yangian $\Y_N$ is equivalent to the usual definition (cf. \cite[Def. 2.1.1]{Mol07}) using the defining relations \eqref{quater u} and \eqref{sym u}, where we have fixed the $N\times N$ symmetric matrix $(g_{ij})$ to be the identity matrix.
Similarly, one can replace the relation \eqref{s11} in Definition~\ref{tydef} by the relation $s_{ii}(u)=s_{ii}(-u)$, for a fixed $1\lle i\lle N$.
\end{rem}

The twisted Yangian $\Y_N$ has an R-matrix presentation (cf. \cite[\S2.2]{Mol07}) via the R-matrix $R(u)$ from \eqref{Ru} and 
\begin{align} \label{Su}
S(u)=\sum_{i,j=1}^N E_{ij}\otimes s_{ij}(u)\in \End(\bC^N)\otimes\Y_N[[u^{-1}]]. 
\end{align} 
Indeed, the quaternary relation \eqref{quater u} and the symmetry relation \eqref{sym u} can be respectively reformulated in the matrix form as
\begin{align}
\label{quat-mat}
R(u-v)S_1(u)R^{t}(-u-v)S_2(v) &=S_2(v)R^{t}(-u-v)S_1(u)R(u-v),
\\
\label{sym-mat}
S^{t}(-u) &=S(u)+\frac{S(u)-S(-u)}{2u}.
\end{align}
Here $t$ stands for the transpose and $R^t(u)$ is obtained from $R(u)$ by taking the transpose on the first (or second) factor.

Set 
\[
\wtl S(u)=S(u)^{-1}=\sum_{i,j=1}^N E_{ij}\otimes \tl s_{ij}(u). 
\]
Then by \eqref{quat-mat}, we have
\beq
\wtl S_2(v) R(u-v)S_1(u)R^t(-u-v)=R^t(-u-v)S_1(u)R(u-v)\wtl S_2(v).
\eeq
In terms of generating series, we have
\beq\label{sts}
\begin{split}
(u^2-v^2)[s_{ij}(u),\tl s_{kl}(v)]= &\, (u+v)\Big(\delta_{jk}\sum_{a=1}^N s_{ia}(u)\tl s_{al}(v)-\delta_{il}\sum_{a=1}^N \tl s_{ka}(v) s_{aj}(u)\Big)\\
 +&\, (u-v)\Big(\delta_{jl}\sum_{a=1}^N \tl s_{ka}(v) s_{ia}(u)-\delta_{ik}\sum_{a=1}^N s_{aj}(u) \tl s_{al}(v)\Big)\\
 &\hskip 0.8cm +\Big(\delta_{ik}\sum_{a=1}^N s_{ja}(u)\tl s_{al}(v)-\delta_{jl}\sum_{a=1}^N \tl s_{ka}(v) s_{ai}(u)\Big).
\end{split}
\eeq
We have the following immediate consequence of \eqref{sts}.

\begin{lem}\label{stslem}
If $\{i,j\}\cap \{k,l\}=\emptyset$, then $[s_{ij}(u),\tl s_{kl}(v)]=0$.
\end{lem}

\subsection{Some basic properties}
It is well known that sending
\[
S(u)\mapsto T(u)T^{t}(-u)
\]
defines an algebra embedding $\Y_N\hookrightarrow \rY(\gl_N)$, and in this way $\Y_N$ can be identified as a subalgebra of $\rY(\gl_N)$ (cf. \cite[\S2.4]{Mol07}). 
Then there is a filtration on $\Y_N$ inherited from the one given by \eqref{filter:Y} on $\rY(\gl_N)$ such that $\deg s_{ij}^{(r)}=r-1$. Let $\mathcal F_s(\Y_N)$ be the subspace of $\Y_N$ spanned by elements of degree $\lle s$ such that
\begin{align}
\label{filter:tY}
    \mathcal F_0(\Y_N) \subset \mathcal F_1(\Y_N) \subset \mathcal F_2(\Y_N) \subset \ldots, 
    \qquad\qquad \Y_N =\bigcup_{s\gge 0} \mathcal F_s(\Y_N).
\end{align}
Denote by $\gr\, \Y_N$ the associated graded algebra. Let $\bar s_{ij}^{(r)}$ be the image of $s_{ij}^{(r)}$ in the $(r-1)$-th component of $\gr\,\Y_N$.

Let $\theta$ be the involution of $\gl_N$ defined by
\[
\theta: \gl_N\longrightarrow \gl_N,\quad e_{ij}\mapsto -e_{ji}.
\]
Extend this to an involution on the current algebra $\gl_N[z]$ (denoted again by $\theta$) sending 
\begin{align*}
    g\otimes z^r \mapsto \theta(g)\otimes(-z)^r \qquad \qquad (g\in \gl_N, \; r\in\bN). 
\end{align*}
Denote by $\gl_N[z]^\theta$ the fixed point subalgebra of $\gl_N[z]$ under the involution $\theta$. Then the map
\[
\mathrm{U}(\gl_N[z]^\theta)\longrightarrow \gr\,\Y_N,\qquad e_{ij}\otimes z^{r}-(-1)^re_{ji}\otimes z^{r} \mapsto \bar s_{ij}^{(r+1)}
\]
induces an algebra isomorphism. 

If $i<j$, then by \eqref{quater u} we have
\beq\label{eq:zero-mode-s}
[s_{ij}(u),s_{j,j+1}^{(1)}]=s_{i,j+1}(u),\qquad [s_{j+1,j}^{(1)},s_{ji}(u)]=s_{j+1,i}(u).
\eeq
There is an anti-automorphism $\eta$ for $\X_N$ (and for $\Y_N$) defined by
\begin{align}
    \label{eta}
\eta:S(u)\longrightarrow S^t(u),\qquad s_{ij}(u)\mapsto s_{ji}(u).
\end{align}
The matrix $\wtl S(-u-\tfrac{N}{2})$ satisfies the quaternary relations \eqref{quat-mat}. Therefore, the map
\begin{align} \label{varpi}
\varpi_N: S(u)\mapsto \wtl S(-u-\tfrac{N}{2})
\end{align}
defines an automorphism $\varpi_N$ of the extended twisted Yangians $\X_N$. Moreover, the automorphism $\varpi_N$ is involutive, i.e., $\varpi_N\circ \varpi_N=1$.

Set 
\begin{align} \label{ii}
    i'=N+1-i. 
\end{align}
Define an automorphism $\rho_N$ of the extended twisted Yangian $\X_N$ by
\[
\rho_N:\X_N\longrightarrow \X_N,\qquad s_{ij}(u)\mapsto s_{i'j'}(u).
\]
By composition, we obtain another automorphism
\beq\label{zeta}
\zeta_N=\rho_N\circ \varpi_N:\X_N\longrightarrow \X_N,\qquad s_{ij}(u)\mapsto \tl s_{i'j'}(-u-\tfrac{N}{2}).
\eeq

Given any $M\in \bN$, we have the natural homomorphism $\jmath$ from $\X_N$ to $\X_{M+N}$ given by
\begin{align}  \label{jmath}
\jmath: \X_N \longrightarrow \X_{M+N},\quad s_{ij}(u)\mapsto s_{ij}(u).
\end{align}
There is also a homomorphism $\iota_M$ from $\X_N$ to $\X_{M+N}$ given by shifting indices,
\begin{align} \label{iota}
\iota_M:\X_{N}\longrightarrow \X_{M+N},\quad s_{ij}(u)\mapsto s_{M+i,M+j}(u).
\end{align}
The following PBW theorem for twisted Yangian $\Y_N$ is also well known (cf. \cite{Mol07}).

\begin{prop}\label{prop:PBW}
Given any linear order on the set of generators $s_{ij}^{(r)}$ and $s_{kk}^{(2r)}$, where $r\in\bZ_{>0}$, $1\lle i,j,k\lle N$ and $i>j$, the ordered monomials in these generators form a basis of $\Y_N$. 
\end{prop}

\section{Gauss decomposition}
\label{sec:GD}

In this section, we introduce Gauss decomposition for the twisted Yangian $\Y_N$ and formulate several basic properties of the new current generators. A PBW basis for $\Y_N$ is given in terms of the  current generators.

\subsection{Quasi-determinants and Gauss decomposition}

Let $X$ be a square matrix over a ring with identity such that its inverse matrix $X^{-1}$ exists, 
and such that its $(j,i)$-th entry is an invertible element of the ring.  Then the $(i,j)$-th
\emph{quasi-determinant} of $X$ is defined by the first formula below and denoted graphically by the boxed notation (cf. \cite{GR97}, \cite[\S1.10]{Mol07}):
\begin{equation*}
\vert X\vert _{ij} \stackrel{\text{def}}{=} \left((X^{-1})_{ji}\right)^{-1} = \left\vert  \begin{array}{ccccc} x_{11} & \cdots & x_{1j} & \cdots & x_{1n}\\
&\cdots & & \cdots&\\
x_{i1} &\cdots &\boxed{x_{ij}} & \cdots & x_{in}\\
& \cdots& &\cdots & \\
x_{n1} & \cdots & x_{nj}& \cdots & x_{nn}
\end{array} \right\vert .
\end{equation*}

By \cite[Theorem 4.96]{GGRW:2005}, the matrix $S(u)$, for both $\X_N$ and $\Y_N$, has the following Gauss decomposition:
$$
S(u) = F(u) D(u) E(u)
$$
for unique matrices of the form
\begin{equation*}
D(u) = \left[ \begin{array}{cccc} d_1 (u) & &\cdots & 0\\
& d_2 (u) &  &\vdots\ \\
\vdots & &\ddots &\\
0 &\cdots &  &d_{N} (u)
\end{array} \right],
\end{equation*}
\begin{equation*}
E(u)=\left[ \begin{array}{cccc} \!\!1 &e_{12}(u) &\cdots & e_{1N}(u)\\
&\ddots & &e_{2N}(u) \!\!\!\!\\
& &\ddots & \vdots\\
0 & & &1 
\end{array} \right],\qquad 
F(u) = \left[ \begin{array}{cccc} \!\!1 & &\cdots &0\!\!\\
f_{21}(u) &\ddots & &\vdots\\
\vdots & & \ddots& \\
\!\!f_{N1}(u) & f_{N2}(u) &\cdots &1\!
\end{array} \right],
\end{equation*}
where the matrix entries are defined in terms of quasi-determinants:
\begin{eqnarray*}
d_i (u) &=& \left\vert  \begin{array}{cccc} s_{11}(u) &\cdots &s_{1,i-1}(u) &s_{1i}(u) \\
\vdots &\ddots & &\vdots \\
s_{i1}(u) &\cdots &s_{i,i-1}(u) &\mybox{$s_{ii}(u)$}
\end{array} \right\vert, 
\qquad \tl d_i(u)=d_i(u)^{-1}, 
\\
e_{ij}(u) &=&\tl d_i (u) \left\vert  \begin{array}{cccc} s_{11}(u) &\cdots & s_{1,i-1}(u) & s_{1j}(u) \\
\vdots &\ddots &\vdots & \vdots \\
s_{i-1,1}(u) &\cdots &s_{i-1,i-1}(u) & s_{i-1,j}(u)\\
s_{i1}(u) &\cdots &s_{i,i-1}(u) &\mybox{$s_{ij}(u)$}
\end{array} \right\vert,
\\
f_{ji}(u) &=& \left\vert  \begin{array}{cccc} s_{11}(u) &\cdots &s_{1, i-1}(u) & s_{1i}(u) \\
\vdots &\ddots &\vdots &\vdots \\
s_{i-1,1}(u) &\cdots &s_{i-1,i-1}(u) &s_{i-1,i}(u)\\
s_{j1}(u) &\cdots &s_{j, i-1}(u) &\mybox{$s_{ji}(u)$} 
\end{array} \right\vert\, 
\tl d_{i}(u).
\end{eqnarray*}
The Gauss decomposition can also be written component-wise as, for $i<j$, 
\begin{align}
s_{ii}(u)&=d_i(u)+\sum_{k<i}f_{ik}(u)d_k(u)e_{ki}(u),\nonumber\\
s_{ij}(u)&=d_i(u)e_{ij}(u)+\sum_{k<i}f_{ik}(u)d_k(u)e_{kj}(u),\label{eq:sij-Gauss}\\
s_{ji}(u)&=f_{ji}(u)d_i(u)+\sum_{k<i}f_{jk}(u)d_k(u)e_{ki}(u).\nonumber
\end{align}

We further denote
\begin{align}
    \label{gauss-gen}
e_{ij}(u) &=\sum_{r\gge 1}e_{ij}^{(r)}u^{-r},\quad f_{ji}(u)=\sum_{r\gge 1}f_{ji}^{(r)}u^{-r},\quad d_k(u)=1+\sum_{r\gge 1}d_{k}^{(r)}u^{-r},
\\
e_i(u) &=\sum_{r\gge 1}e_{i}^{(r)}u^{-r}=e_{i,i+1}(u),\quad f_i(u)=\sum_{r\gge 1}f_{i}^{(r)}u^{-r}=f_{i+1,i}(u),\quad 1\lle i<N.
\end{align}

Set 
\beq\label{edfinv}
\begin{split}
\wtl D(u)&=D(u)^{-1}=\sum_{1\lle i\lle N} E_{ii}\otimes \tl d_{i}(u),
\\
\wtl E(u)&=E(u)^{-1}=1+\sum_{1\lle i<j\lle N}E_{ij}\otimes \tl e_{ij}(u),\\
\wtl F(u)&=F(u)^{-1}=1+\sum_{1\lle i<j\lle N}E_{ji}\otimes \tl f_{ji}(u).
\end{split}
\eeq
Then we have 
\beq\label{eq:def-tilde-e-f}
\begin{split}
&\tl e_{ij}(u)=\sum_{i=i_0<i_1<\cdots<i_s=j}(-1)^s e_{i_0i_1}(u)e_{i_1i_2}(u)\cdots e_{i_{s-1}i_s}(u),\\
&\tl f_{ji}(u)=\sum_{i=i_0<i_1<\cdots<i_s=j}(-1)^s f_{i_{s}i_{s-1}}(u)\cdots f_{i_2i_1}(u) f_{i_1i_0}(u).
\end{split}
\eeq

\subsection{Properties of new current generators}
Recall $\varpi_N$ and $\iota_M$ from \eqref{varpi} and \eqref{iota}. Define an embedding $\vartheta_M:\X_N \to \X_{M+N}$ 
as the composition
\beq\label{eq:theta-m}
\vartheta_M= \varpi_{M+N}\circ \iota_M\circ \varpi_N \colon \X_N \longrightarrow \X_{M+N}.
\eeq
\begin{lem}[{\cite[Lemma 2.14.1]{Mol07}}]\label{lem:theta}
For any $1\lle i,j\lle N$, $\vartheta_M$ maps $s_{ij}(u+\tfrac{M}{2})$ to
\[
\left\vert  \begin{array}{cccc} s_{11}(u) &\cdots &s_{1,M}(u) &s_{1, M+j}(u)\\
\vdots &\ddots &\vdots &\vdots \\
s_{M1}(u) &\cdots &s_{MM}(u) &s_{M, M+j}(u)\\
s_{M+i, 1}(u) &\cdots &s_{M+i,M}(u) &\mybox{$s_{M+i,M+j}(u)$}
\end{array} \right\vert.
\]
\end{lem}

\begin{cor}\label{theta}
The map $\vartheta_M\colon \X_N \longrightarrow \X_{M+N}$ sends 
\[
d_i(u+\tfrac{M}{2})\mapsto d_{M+i}(u), \quad
e_{ij}(u+\tfrac{M}{2})\mapsto e_{M+i,M+j}(u), \quad
f_{ji}(u+\tfrac{M}{2})\mapsto f_{M+j,M+i}(u).
\]
\end{cor}

\begin{proof}
By tracing the definition $\vartheta_M$ in \eqref{eq:theta-m} through \eqref{varpi} and \eqref{iota}, one sees that $\vartheta_M$ maps $\tl s_{ij}(-u-\frac{N}{2})$ to $\tl s_{M+i,M+j}(-u-\frac{M+N}{2})$. It then follows that 
\begin{align} \label{l12}
    \vartheta_{\ell_1}\circ \vartheta_{\ell_2}=\vartheta_{\ell_1+\ell_2}, \qquad (\ell_1,\ell_2\gge 0); 
\end{align}
here and below it is understood that $N$ is fixed throughout when considering $\vartheta_M$ for various $M$. By Lemma \ref{lem:theta}, we have
\begin{align*}
d_i(u)&=\vartheta_{i-1}\big(s_{11}(u+\tfrac{i-1}{2})\big)\\
e_{ij}(u)&=\vartheta_{i-1}(s_{11}\big(u+\tfrac{i-1}{2})^{-1}s_{1,j-i+1}(u+\tfrac{i-1}{2})\big),\\
f_{ji}(u)&=\vartheta_{i-1}\big(s_{j-i+1,1}(u+\tfrac{i-1}{2})s_{11}(u+\tfrac{i-1}{2})^{-1}\big).
\end{align*}
Then the lemma follows from applying \eqref{l12}.
\end{proof}

Recall the homomorphism $\jmath: \X_M \rightarrow \X_{M+N}$ from \eqref{jmath}.

\begin{lem}\label{comsubalg}
The subalgebras $\jmath(\X_M)$ and $\vartheta_M(\X_N)$ of $\X_{M+N}$ commute with each other.
\end{lem}
\begin{proof}
Note that the map $\vartheta_M$ maps $\tl s_{ij}(-u-\frac{N}{2})$ in $\X_{N}$ to $\tl s_{M+i,M+j}(-u-\frac{M+N}{2})$ in $\X_{M+N}$. Thus $\vartheta_M(\X_N)$ is generated by the coefficients of $\tl s_{M+i,M+j}(u)$ for $1\lle i,j\lle N$. Now the statement follows immediately from Lemma \ref{stslem}.
\end{proof}

As an immediate consequence, we obtain some simple relations among the new generators. 

\begin{cor}\label{cor:commu-d-e-f}
We have $[d_i(u),d_j(u)]=0$ for all $i\ne j$, and
\begin{align*}
&[e_i(u),e_j(v)]=[e_i(u),f_j(v)]=[f_i(u),f_j(v)]=0, &\text{for }|i-j|>1,\\
&[d_i(u),e_{j}(v)]=[d_i(u),f_j(v)]=0,&\text{ for }i\ne j,j+1.
\end{align*}
\end{cor}

Recall the anti-involution $\eta$ from \eqref{eta}.
\begin{lem}\label{lem:tau}
For $1\lle i<j\lle N$ and $\ 1\lle k\lle N$, we have
\[
\eta\big(e_{ij}(u)\big)=f_{ji}(u),\qquad \eta\big(f_{ji}(u)\big)=e_{ij}(u),\qquad \eta\big(d_{k}(u)\big)=d_{k}(u).
\]
\end{lem}

Recall the automorphism $\zeta_N$ from \eqref{zeta} and $\tl e_{ij}(u), \tl f_{ij}(u)$ from \eqref{edfinv}--\eqref{eq:def-tilde-e-f}.
\begin{lem}\label{zetamap}
For $1\lle i<j\lle N$ and $\ 1\lle k\lle N$, we have
\[
\zeta_N(e_{ij}(u))=\tl f_{i'j'}(-u-\tfrac{N}{2}),\quad \zeta_N(f_{ji}(u))=\tl e_{j'i'}(-u-\tfrac{N}{2}),\quad \zeta_N\big(d_{k}(u)\big)=\tl d_{k'}(-u-\tfrac{N}{2}).
\]
In particular, $\zeta_N\big(e_{i}(u)\big)=-f_{N-i}(-u-\tfrac{N}{2})$ and $
\zeta_N\big(f_{i}(u)\big)=-e_{N-i}(-u-\tfrac{N}{2})$.
\end{lem}
\begin{proof}
Multiplying out the product $T(u)^{-1}=E(u)^{-1}D(u)^{-1}F(u)^{-1}$, we have
\begin{align*}
\tl s_{ii}(u)&=\tl d_i(u)+\sum_{k>i}\tl e_{ik}(u)\tl d_k(u)\tl f_{ki}(u),\\
\tl s_{ij}(u)&=\tl e_{ij}(u)\tl d_j(u)+\sum_{k>j}\tl e_{ik}(u)\tl d_k(u)\tl f_{kj}(u),\\
\tl s_{ji}(u)&=\tl d_j(u)\tl f_{ji}(u)+\sum_{k>j}\tl e_{jk}(u)\tl d_k(u)\tl f_{ki}(u).
\end{align*}
Note that $\zeta_N$ is an automorphism and $\zeta_N(s_{ij}(u))=\tl s_{i'j'}(-u-\frac{N}{2})$, the claim follows from the uniqueness of Gauss decomposition by comparing the above formulas with \eqref{eq:sij-Gauss}.
\end{proof}

\begin{lem}\label{lem:ei-generate-eij}
For $1\lle i<j< N$, we have
\beq\label{eq:ei-generate-eij}
e_{i,j+1}(u)=[e_{ij}(u),e_{j,j+1}^{(1)}],\qquad f_{j+1,i}(u)=[f_{j+1,j}^{(1)},f_{ji}(u)].
\eeq
In particular, the algebra $\Y_N$ \emph{(}resp. $\X_N$\emph{)} is generated by coefficients of $d_i(u)$, $e_j(u)$, and $f_j(u)$, for $1\lle i\lle N$ and $1\lle j<N$.
\end{lem}
\begin{proof}
We prove the first identity by induction on $i$. Note that $e_{j,j+1}^{(1)}=s_{j,j+1}^{(1)}$. By \eqref{eq:sij-Gauss} and \eqref{eq:zero-mode-s}, we have
\[
\big[d_1(u)e_{1j}(u),e_{j,j+1}^{(1)}\big]=d_1(u)e_{1,j+1}(u).
\]
The base case $i=1$ follows from this and the fact that the (invertible) $d_1(u)$ commutes with $e_{j,j+1}^{(1)}$.

Again by \eqref{eq:sij-Gauss} and \eqref{eq:zero-mode-s}, we have
\begin{align} \label{expand ss}
\Big[d_i(u)e_{ij}(u)+\sum_{k<i}f_{ik}(u)d_k(u)e_{kj}(u),e_{j,j+1}^{(1)}\Big]=d_i(u)e_{i,j+1}(u)+\sum_{k<i}f_{ik}(u)d_k(u)e_{k,j+1}(u).
\end{align}
By Lemma \ref{comsubalg}, for $k<i<j$, we have
\[
[d_i(u),e_{j,j+1}^{(1)}]=[d_k(u),e_{j,j+1}^{(1)}]=[f_{ik}(u),e_{j,j+1}^{(1)}]=0.
\]
By these identities together with the induction hypothesis $[e_{kj}(u),e_{j,j+1}^{(1)}]=e_{k,j+1}(u)$, we simplify \eqref{expand ss} to become
$
d_i(u)[e_{ij}(u),e_{j,j+1}^{(1)}]=d_i(u)e_{i,j+1}(u),
$ 
proving the desired identity in \eqref{eq:ei-generate-eij}. 

The second identity in \eqref{eq:ei-generate-eij} follows from the first one by applying the anti-automorphism $\eta$. 
\end{proof}

\subsection{Special twisted Yangians}
Let $g(u)$ be any even formal power series in $u^{-1}$ with leading term $1$,
\[
g(u)=1+g_1u^{-2}+g_2u^{-4}+\cdots\in \bC[[u^{-2}]].
\]
There is an automorphism of $\Y_N$ defined by
\beq\label{eq:mu_g-twisted}
\nu_{g(u)}:\Y_N \longrightarrow \Y_N, \qquad S(u)\mapsto g(u)S(u).
\eeq

Recall the Yangian $\rY(\mathfrak{sl}_N)$ from \S\ref{subsec:Y}. The \textit{special twisted Yangian} $\SY_N$ is by definition the subalgebra of $\Y_N$:
\[
\SY_N=\rY(\mathfrak{sl}_N)\cap \Y_N.
\]
Alternatively, $\SY_N$ is the subalgebra of $\Y_N$ which consists of the elements that are stable under all automorphisms of the form \eqref{eq:mu_g-twisted}.

Define $\sdet\, S(u)=d_1(u)d_2(u-1)\cdots d_N(u-N+1)$ and denote by $\mathscr{ZY}_N$ the center of $\Y_N$. Set
\beq\label{sdet}
\sdet\, S(u)=1+\sum_{r\gge 1} c_r u^{-r}.
\eeq

\begin{prop}\label{prop:center}
\cite[Theorems 2.8.2, 2.9.2, 2.12.1]{Mol07}
\qquad
\begin{enumerate}
    \item 
The coefficients $c_2,c_4,\ldots$ of $\sdet\, S(u)$ are algebraically independent generators of the center $\mathscr{ZY}_N$ of $\Y_N$.

\item 
We have an algebra isomorphism
$\Y_N\cong \mathscr{ZY}_N\otimes \SY_N.$
Moreover, the center of $\SY_N$ is trivial.
\end{enumerate}
\end{prop}

The special twisted Yangian $\SY_N$ can also be regarded as a quotient of the twisted Yangian $\Y_N$.
\begin{lem}\label{syquo}
The subalgebra $\SY_N$ is isomorphic to the quotient of $\Y_N$ by the ideal generated by the coefficients of $\sdet\,S(u)-1$, i.e., 
$\SY_N\cong \Y_N/(\sdet\,S(u)-1).$
\end{lem}

For later use, we shall define the following generating series, for $1\lle i<j \lle N$:
\begin{align}
     \label{bdef}
\begin{split}
b_{ji}(u)&=\sum_{r\gge 0}b_{ji,r}u^{-r-1}:=f_{ji}(u-\tfrac{i}{2}),\\b_i(u)&=\sum_{r\gge 0}b_{i,r}u^{-r-1}:= b_{i+1,i}(u),
\end{split}
\\
\label{hdef}
\begin{split}
h_0(u)&=1+\sum_{r\gge 0}h_{0,r}u^{-r-1}:=d_1(u),\\
h_i(u)&=1+\sum_{r\gge 0}h_{i,r}u^{-r-1}:=\tl d_{i}(u-\tfrac{i}{2})d_{i+1}(u-\tfrac{i}{2}). 
\end{split}
\end{align}

In terms of the new notations in \eqref{bdef}--\eqref{hdef}, Corollary~\ref{theta} can be rephrased as follows.

\begin{cor} \label{theta2}
The map $\vartheta_M\colon \X_N \longrightarrow \X_{M+N}$ sends 
    $b_i(u)\mapsto b_{M+i}(u)$ and $h_i(u)\mapsto h_{M+i}(u)$.
\end{cor}
\begin{lem}\label{lem:special-in-Gauss}
The special twisted Yangian $\SY_N$ is generated by the coefficients of  $e_{i}(u)$, $f_{i}(u)$, and $h_i(u)$, for $1\lle i<N$.
\end{lem}
\begin{proof}
Let $\mathcal{SY}_N$ be the subalgebra of $\Y_N$ generated by the coefficients of  $e_{i}(u)$, $f_{i}(u)$, and $h_i(u)$, for $1\lle i< N$. Clearly, $e_{ij}(u)$, $f_{ji}(u)$, and $h_i(u)$, for $1\lle i<j\lle N$, are stable under all
automorphisms of the form \eqref{eq:mu_g-twisted}. Therefore, we have $\mathcal{SY}_N\subset \SY_N$.

Using
$$\sdet\, S(u)=d_1(u)d_2(u-1)\cdots d_{N}(u-N+1),\qquad h_i(u+\tfrac{i}{2})=\tl d_i(u)d_{i+1}(u)
$$
and Proposition \ref{prop:center}, we have that $d_1(u)\in \mathscr{ZY}_N\cdot \mathcal{SY}_N[[u^{-1}]]$. Hence $d_i(u)\in \mathscr{ZY}_N\cdot \mathcal{SY}_N[[u^{-1}]]$ for all $1\lle i\lle N$. Note that $e_{ij}(u)$ and $f_{ji}(u)$ for $1\lle i<j\lle N$ are in $\mathcal{SY}_N$ by Lemma \ref{lem:ei-generate-eij}. Therefore $\mathscr{ZY}_N\cdot \mathcal{SY}_N=\Y_N$. The claim now follows from Proposition \ref{prop:center}.
\end{proof}

In light of Lemma \ref{lem:special-in-Gauss}, we define the special extended Yangian as follows.

\begin{dfn}\label{sxdef}
The \emph{special extended Yangian} $\SX_N$ is the subalgebra of $\X_N$ generated by the coefficients of $e_{i}(u)$, $f_{i}(u)$, and $h_i(u)$, for $1\lle i< N$.
\end{dfn}

Below we formulate the counterparts of Proposition \ref{prop:PBW} on the PBW type bases for $\Y_N$ and $\SY_N$.

\begin{lem}\label{PBWlem}
Given any linear order on the set of generators $f_{ji}^{(r)}$ and $d_{i}^{(2r)}$, where $r\in\bZ_{> 0}$, $1\lle i<j\lle N$, the ordered monomials in these generators form a basis of $\Y_N$. 
\end{lem}
\begin{proof}
Recall the filtration  $\{\mathcal F_s \Y_N\}_{s\gge 0}$ from \eqref{filter:tY}. By Gauss decomposition \eqref{eq:sij-Gauss}, we have
\begin{align*}
f_{ji}^{(r)}& \equiv s_{ji}^{(r)} \ \mod \mathcal F_{r-2}\Y_N,\\
d_{i}^{(2r)}&\equiv s_{ii}^{(2r)} \mod \mathcal F_{2r-2}\Y_N.
\end{align*}
Then the statement is an immediate corollary of Proposition \ref{prop:PBW}.
\end{proof}

\begin{prop}\label{PBWgauss}
Given any linear order on the set of generators $b_{ji,r}$ and $h_{k,2s+1}$, where $r,s\in\bN$, $1\lle i<j\lle N$, $0\lle k<N$ (resp. $1\lle k<N$), the ordered monomials in these generators form a basis of $\Y_N$ (resp. $\SY_N$). 
\end{prop}
\begin{proof}
By Gauss decomposition and the definitions of $b_{ji,r}$ and $h_{i,2r+1}$ from \eqref{bdef} and \eqref{hdef}, we have that
\begin{align*}
&b_{ji,r}\equiv f_{ji}^{(r+1)}\equiv s_{ji}^{(r+1)} \qquad \qquad \qquad \qquad \quad \ \ \, \mod \mathcal F_{r-1}\Y_N,\\
&h_{0,2r+1}\equiv d_{1}^{(2r+2)} \equiv s_{11}^{(2r+2)} \qquad \qquad \ \,\qquad \quad \ \ \, \mod \mathcal F_{2r}\Y_N,\\
&h_{i,2r+1}\equiv d_{i+1}^{(2r+2)}-d_i^{(2r+2)}\equiv s_{i+1,i+1}^{(2r+2)}-s_{ii}^{(2r+2)} \mod \mathcal F_{2r}\Y_N.
\end{align*}
Recall the generators $c_{2r}$ of the center $\mathscr Z\Y_N$ from \eqref{sdet}. It is known from the proof of \cite[Thm~ 2.8.2]{Mol07} that
$c_{2r}\equiv \sum_{i=1}^N d_i^{(2r)} \mod \mathcal F_{2r-2}\Y_N.$
Hence the statement follows from Proposition \ref{prop:center} and Lemma \ref{PBWlem}.
\end{proof}

\section{Current relations for twisted Yangians of low ranks}
\label{sec:low rank}

In this section, we establish new relations among the new current generators for (extended/special) twisted Yangians of rank one and two.

\subsection{The rank 1 case}

In this subsection, we shall work with $\X_2$ and $\SX_2$. Recall 
\begin{align} \label{GD2}
\begin{split}
S(u) =\begin{bmatrix} 1 & 0\\ f(u) & 1\end{bmatrix}\begin{bmatrix} d_1(u) & 0\\ 0 & d_2(u)\end{bmatrix}\begin{bmatrix} 1 & e(u)\\ 0 & 1\end{bmatrix}
&=\begin{bmatrix} d_1(u) & d_1(u)e(u)\\ f(u)d_1(u) & d_2(u)+f(u)d_1(u)e(u)\end{bmatrix},
\\
\wtl S(u) =\begin{bmatrix} 1 & -e(u)\\ 0 & 1\end{bmatrix}\begin{bmatrix} \tl d_1(u) & 0\\ 0 & \tl d_2(u)\end{bmatrix}\begin{bmatrix} 1 & 0\\ -f(u) & 1\end{bmatrix}
&=\begin{bmatrix} \tl d_1(u)+e(u)\tl d_2(u)f(u) & -e(u)\tl d_2(u)\\ -\tl d_2(u)f(u) & \tl d_2(u)\end{bmatrix},
\end{split}
\end{align}
where we dropped the subscripts for $e_1(u)$ and $f_1(u)$ in this rank 1 case. 
In this case, we have that 
\begin{align} \label{sdet2}
    \mathrm{sdet}\,S(u)=d_1(u)d_2(u-1).
\end{align}

\begin{lem}\label{lem:o2-relations}
The following relations hold in $\X_2$:
\begin{align}
[d_i(u),d_j(v)] &=0,\label{eq:o2-1}\\
f(u) &=e(-u-1),\label{eq:o2-2}\\
\tl d_1(u)d_2(u) &=\tl d_1(-u-1)d_2(-u-1),\label{eq:o2-3}\\
[d_1(u),e(v)] &=\frac{1}{u-v}d_1(u)\big(e(v)-e(u)\big)+\frac{1}{u+v+1}\big(f(u)-e(v)\big)d_1(u),\label{eq:o2-4}\\
[d_1(u),f(v)] &=\frac{1}{u-v}\big(f(u)-f(v)\big)d_1(u)+\frac{1}{u+v+1}d_1(u)\big(f(v)-e(u)\big),\label{eq:o2-5}\\
[d_2(u),e(v)] &=\frac{1}{u-v}d_2(u)\big(e(u)-e(v)\big)+\frac{1}{u+v+1}\big(e(v)-f(u)\big)d_2(u),\label{eq:o2-6}\\
[d_2(u),f(v)] &=\frac{1}{u-v}\big(f(v)-f(u)\big)d_2(u)+\frac{1}{u+v+1}d_2(u)\big(e(u)-f(v)\big),\label{eq:o2-7}\\
[e(u),f(v)] &=\frac{1}{u-v}\big(\tl d_1(u)d_2(u)-\tl d_1(v)d_2(v)\big)+\frac{1}{u+v+1}\big(e(u)-f(v)\big)^2,\label{eq:o2-8}\\
[e(u),e(v)] &=\frac{1}{u-v}\big(e(u)-e(v)\big)^2+\frac{1}{u+v+1}\big(\tl d_1(u)d_2(u)-\tl d_1(v)d_2(v)\big),\label{eq:o2-9}\\
[f(u),f(v)] &=-\frac{1}{u-v}\big(f(u)-f(v)\big)^2-\frac{1}{u+v+1}\big(\tl d_1(u)d_2(u)-\tl d_1(v)d_2(v)\big).\label{eq:o2-10}
\end{align}
\end{lem}
\begin{proof} We proceed to verify these relations one-by-one. 

\mybox{Equation \eqref{eq:o2-1}}. Setting $i=j=k=l=1$ in the quaternary relation \eqref{quater u}, we obtain 
$$
(u-v+1)(u+v+1)[s_{11}(u),s_{11}(v)]=0$$ and hence $[s_{11}(u),s_{11}(v)]=0$, i.e., $[d_1(u),d_1(v)]=0$. Since the coefficients of $\mathrm{sdet}\,S(u)$ are central and recall \eqref{sdet2}, we deduce that $[d_i(u),d_j(v)]=0$. (Note that $[d_1(u),\tl d_2(v)]=0$ also follows from Lemma \ref{stslem}.)

\mybox{Equations \eqref{eq:o2-4}--\eqref{eq:o2-5}}. Equation \eqref{eq:o2-5} follows from \eqref{eq:o2-4} by applying the anti-automorphism $\eta$ and  Lemma \ref{lem:tau}. It remains to prove \eqref{eq:o2-4}. Setting $i=j=k=1$ and $l=2$ in \eqref{sts}, one has
\begin{align*}
(u^2-v^2)&[s_{11}(u),\tl s_{12}(v)]\\
=&\,(u+v+1)\big(s_{11}(u)\tl s_{12}(v)+s_{12}(u)\tl s_{22}(v)\big)-(u-v)\big(s_{11}(u)\tl s_{12}(v)+s_{21}(u)\tl s_{22}(v)\big).
\end{align*}
This can be converted by \eqref{GD2} to be
\begin{align*}
(u^2-v^2)[d_1(u),-e(v)\tl d_2(v)]=(u+v+1)&\big(-d_1(u)e(v)\tl d_2(v)+d_1(u)e(u)\tl d_2(v) \big)\\
 -(u-v)&\big(-d_1(u)e(v)\tl d_2(v)+f(u)d_1(u)\tl d_2(v) \big).
\end{align*}
Multiplying both sides by $d_2(v)$ from the right, we obtain
\begin{align*}
(u^2-v^2)&[d_1(u),e(v)]\\=&\,(u+v+1)d_1(u)\big(e(v)-e(u)\big)-(u-v)\big(d_1(u)e(v)-f(u)d_1(u)\big).
\end{align*}
This identity can then be rewritten as 
\[
(u-v)(u+v+1)[d_1(u),e(v)]=(u+v+1)d_1(u)\big(e(v)-e(u)\big)-(u-v)\big(e(v)-f(u)\big)d_1(u),
\]
proving \eqref{eq:o2-4}. 

\mybox{Equation \eqref{eq:o2-2}}. Multiplying both sides of \eqref{eq:o2-4} by $u+v+1$
 and then setting $v=-u-1$, we obtain that $\big(f(u)-e(-u-1)\big)d_1(u)=0$. Hence $f(u)=e(-u-1)$.
 
\mybox{Equation \eqref{eq:o2-6}}. 
Applying the automorphism $\zeta_2$ to both sides of \eqref{eq:o2-5} with the help of Lemma \ref{zetamap} and then changing variables $u\mapsto -u-1, v\mapsto -v-1$, we have
\[
[\tl d_2(u),e(v)] =\frac{1}{-u+v}\big(e(u)-e(v)\big)\tl d_2(u)+\frac{1}{-u-v-1}\tl d_2(u)\big(e(v)-f(u)\big).
\]
Multiplying by $d_2(u)$ on both the left and right of the above identity gives us \eqref{eq:o2-6}. 

\mybox{Equation \eqref{eq:o2-7}}. Equation \eqref{eq:o2-7} can be derived from \eqref{eq:o2-4} in the same way as \eqref{eq:o2-6} is derived from \eqref{eq:o2-5} above. We skip the detail.

\mybox{Equation \eqref{eq:o2-8}}. 
We rewrite \eqref{eq:o2-5} as
\beq\label{o2-5b}
\begin{split}
-(u-v)f(v)d_1(u)&+\big(f(v)-f(u)\big)d_1(u)\\=&-(u-v)d_1(u)f(v)-\frac{u-v}{u+v+1}d_1(u)\big(e(u)-f(v)\big).
\end{split}
\eeq
We also reformulate \eqref{eq:o2-6} as 
\beq\label{o2-6b}
[e(u),\tl d_2(v)]=\frac{1}{u-v}\big(e(u)-e(v)\big)\tl d_2(v)+\frac{1}{u+v+1}\tl d_2(v)\big(e(u)-f(v)\big),
\eeq
which implies that
\beq\label{o2-6c}
(u-v)e(u)\tl d_2(v)+\big(e(v)-e(u)\big)\tl d_2(v)=(u-v)\tl d_2(v)e(u)+\frac{u-v}{u+v+1}\tl d_2(v)\big(e(u)-f(v)\big).
\eeq
Applying \eqref{sts} with $i=l=1$ and $j=k=2$, we have
\begin{align*}
(v-u)\tl s_{21}(v)s_{12}(u)+&\,\tl s_{21}(v)s_{12}(u)+\tl s_{22}(v)s_{22}(u)\\
=&\,(v-u)s_{12}(u)\tl s_{21}(v)+s_{11}(u)\tl s_{11}(v)+s_{12}(u)\tl s_{21}(v),
\end{align*}
which is equivalent to
\begin{align*}
(u-v)&\, \tl d_2(v)f(v)d_1(u)e(u)-\tl d_2(v)f(v)d_1(u)e(u)+\tl d_2(v)\big(d_2(u)+f(u)d_1(u)e(u)\big)\\
=&\, (u-v)d_1(u)e(u)\tl d_2(v)f(v)+d_1(u)\big(\tl d_1(v)+e(v)\tl d_2(v)f(v)\big)-d_1(u)e(u)\tl d_2(v)f(v).
\end{align*}
We rewrite it further as
\begin{align} \label{i1j2}
\begin{split}
    \tl d_2(v)\Big((u-v)f(v)&d_1(u) -\big(f(v)-f(u)\big)d_1(u)\Big)e(u)+\tl d_2(v)d_2(u)\\
=\, &d_1(u)\Big((u-v)e(u)\tl d_2(v)+\big(e(v)-e(u)\big)\tl d_2(v)\Big)f(v)+d_1(u)\tl d_1(v).
\end{split}
\end{align}

Substituting \eqref{o2-5b} and \eqref{o2-6c} on the left- and right-hand sides of \eqref{i1j2} and then multiplying by $\tl d_1(u)d_2(v)$ on the left, we obtain that
\begin{align*}
(u-v)f(v)e(u)+&\,\frac{u-v}{u+v+1}\big(e(u)-f(v)\big)e(u)+\tl d_1(u)d_2(u)\\
=&\, (u-v)e(u)f(v)+\frac{u-v}{u+v+1}\big(e(u)-f(v)\big)f(v)+\tl d_1(v)d_2(v),
\end{align*}
proving \eqref{eq:o2-8}.

\mybox{Equation \eqref{eq:o2-3}}. 
Since $e(u)-f(v)$ has a zero at $v=-u-1$ by \eqref{eq:o2-2}, $\frac{1}{u+v+1}\big(e(u)-f(v)\big)^2$ approaches 0 as $v$ tends to $-u-1$. Then \eqref{eq:o2-8} with $v$ tending to $-u-1$ gives us \eqref{eq:o2-3}.

\mybox{Equations \eqref{eq:o2-9}--\eqref{eq:o2-10}}.  
Replacing $v$ by $-v-1$ in \eqref{eq:o2-8} and then using $f(-v-1)=e(v)$, we prove \eqref{eq:o2-9}. Equation~\eqref{eq:o2-10} follows from \eqref{eq:o2-9} by applying the anti-automorphism $\eta$ and Lemma~ \ref{lem:tau}.
\end{proof}

 Introduce the following series 
\begin{align*}
    b(u)&:=\sum_{r\gge 0} b_ru^{-r-1}=f(u-\tfrac12),
    \\
h(u)&:=  1+ \sum_{r\gge 0} h_ru^{-r-1} =\tl d_1(u-\tfrac12)d_2(u-\tfrac12).
\end{align*}
Due to \eqref{eq:o2-2}, we have $b(u)=e(-u-\tfrac12)$. It follows from \eqref{eq:o2-3} that
$$h(u)=h(-u),$$
which is equivalent to that $h_{2r}=0$, for all $r \in \bN$. 
\begin{lem}\label{sx2l}
We have the following relations in $\X_2$ in generating series:
\begin{align}
[h(u),h(v)] &=0,\label{hhu}\\
(u-v)[b(u),b(v)] &=-\big(b(u)-b(v)\big)^2-\frac{u-v}{u+v}\big(h(u)-h(v)\big),\label{bbu}\\
(u^2-v^2)[h(u),b(v)] &=(2v+1)h(u)b(v)+(2v-1)b(v)h(u)\label{hbu}\\
 +(u-v+1)b(-u) &h(u) -2(u+v)\tl d_1(u-\tfrac12)b(u)d_2(u-\tfrac12)+(u-v-1)h(u)b(-u).\nonumber
\end{align}
\end{lem}
\begin{proof}
The first two relations \eqref{hhu}--\eqref{bbu} follow directly from \eqref{eq:o2-1} and \eqref{eq:o2-10} in Lemma \ref{lem:o2-relations}. It remains to prove the relation \eqref{hbu}.

It follows from \eqref{eq:o2-4} that
\beq\label{proof7}
[\tl d_1(u),e(v)]=\frac{1}{u-v}\big(e(u)-e(v)\big)\tl d_1(v)+\frac{1}{u+v+1}\tl d_1(u)\big(f(u)-e(v)\big).
\eeq
Combining this with  \eqref{eq:o2-6}, we have
\beq\label{proof6}
\begin{split}
2\tl d_1(u)e(v)d_2(u)-&\,\tl d_1(u)d_2(u)e(v)-e(v)\tl d_1(u)d_2(u)\\
=&\, [\tl d_1(u),e(v)]d_2(u)-\tl d_1(u)[d_2(u),e(v)]\\
=&\, \frac{1}{u-v}\big(e(u)-e(v)\big)\tl d_1(u)d_2(u)-\frac{1}{u-v}\tl d_1(u)d_2(u)\big(e(u)-e(v)\big).
\end{split}
\eeq
Below it suffices for us to focus only on the terms containing $e(v)$. By \eqref{eq:o2-6}, \eqref{proof7}, and \eqref{proof6},  we have
\begin{align*}
[\tl d_1(u)& d_2(u),e(v)]=[\tl d_1(u),e(v)]d_2(u)+\tl d_1(u)[d_2(u),e(v)]\\
=&\, \frac{1}{u-v}\big(e(u)-e(v)\big)\tl d_1(u)d_2(u)+\frac{1}{u+v+1}\tl d_1(u)\big(e(v)-f(u)\big)d_2(u)\\
+&\, \frac{1}{u-v}\tl d_1(u)d_2(u)\big(e(u)-e(v)\big)+\frac{1}{u+v+1}\tl d_1(u)\big(e(v)-f(u)\big)d_2(u)\\
=&\, \frac{1}{u-v}\big(e(u)-e(v)\big)\tl d_1(u)d_2(u)+\frac{1}{u+v+1}e(v)\tl d_1(u)d_2(u)-\frac{1}{u+v+1}\tl d_1(u)f(u)d_2(u)\\+&\,\frac{1}{u-v}\tl d_1(u)d_2(u)\big(e(u)-e(v)\big)+\frac{1}{u+v+1}\tl d_1(u)d_2(u)e(v)-\frac{1}{u+v+1}\tl d_1(u)f(u)d_2(u)\\
&\qquad \qquad + \frac{1}{u+v+1}\Big(\frac{1}{u-v}\big(e(u)-e(v)\big)\tl d_1(u)d_2(u)-\frac{1}{u-v}\tl d_1(u)d_2(u)\big(e(u)-e(v)\big)\Big)\\
=&\, \frac{1}{{(u-v)(u+v+1)}}\Big(-(2v+2)e(v)\tl d_1(u)d_2(u)-2v\tl d_1(u)d_2(u)e(v)\\
&\quad\quad\, +(u+v+2)e(u)\tl d_1(u)d_2(u)-2(u-v)\tl d_1(u)f(u)d_2(u)+(u+v)\tl d_1(u)d_2(u)e(u)\Big).
\end{align*}
Note that $b(u)=e(-u-\tfrac12)$ and $h(u)=\tl d_1(u-\tfrac12)d_2(u-\tfrac12)$, substituting $u$ and $v$ by $u-\tfrac12$ and $-v-\tfrac12$ in the above equation proves the  relation \eqref{hbu}.
\end{proof}

Next we work out the relations for $\SX_2$. By definition, the subalgebra $\SX_2$ is generated by the coefficients of $b(u)$ and $h(u)$.

Throughout the paper, we shall use the following notation: 
\[
\{x,y\}=xy+yx,
\]
for $x, y$ in some algebra. Set $h_{-1}=1$.

\begin{prop}\label{sx2p}
The following relations hold in $\SX_2$ for $r,s\in\bN$:
\begin{align}
[h_r,h_s] &=0,\qquad\qquad h_{2r}=0,\label{hh}\\
[b_{r+1},b_s]-[b_r,b_{s+1}] &=\{b_r,b_s\}-2(-1)^rh_{r+s+1},\label{bb}\\
[h_{r+1},b_s]-[h_{r-1},b_{s+2}] &=2\{h_{r-1},b_{s+1}\}+[h_{r-1},b_s].\label{hb}
\end{align}
\end{prop}
\begin{proof} 
The relations \eqref{hh} and \eqref{hb} follow from \eqref{hhu} and \eqref{hbu}, respectively, by comparing the coefficients of $u^{-r-1}v^{-s-1}$ and $u^{-r}v^{-s-1}$. The relation \eqref{bb} is obtained from \eqref{bbu} using the expansion 
$\frac{u-v}{u+v}=-1+2\sum_{p\gge 0}(-1)^p v^pu^{-p}.$
\end{proof}
It will be proved later in greater generality (see Theorem \ref{mainthm2}) that the relations \eqref{hh}--\eqref{hb} are defining relations for $\SX_2$. 

\begin{rem}
We can rewrite \eqref{bbu} as
$$
[b(u),b(v)]=-\frac{1}{u-v}(b(u)-b(v))^2-\frac{1}{u+v}(h(u)-h(v)),
$$
which gives rise to the following alternative of \eqref{bb}:
\beq\label{bibi}
[b_r,b_s]=\sum_{p=0}^{r-1}b_{r+s-1-p}b_{p}-\sum_{p=0}^{s-1}b_{r+s-1-p}b_{p}+(-1)^rh_{r+s}.
\eeq
\end{rem}

\subsection{The rank 2 case}

In the rank 2 case (i.e., $N=3$), we have the following expressions with ``$(u)$" omitted for the matrices $S(u)$ and $\tl S(u)$ and their entries:
\begin{align} \label{SS3}
\begin{split}
    S&=\begin{bmatrix}
d_1 & d_1e_1& d_1e_{13}\\
f_1d_1 & d_2+f_1d_1e_1& d_2e_2+f_1d_1e_{13}\\
f_{31}d_1 & f_2d_2+f_{31}d_1e_{1} & d_3+f_2d_2e_2+f_{31}d_1e_{13}
\end{bmatrix},
\\
\wtl S&=\begin{bmatrix}
\tl d_1 +e_1\tl d_2 f_1+\tl e_{13}\tl d_3 \tl f_{31} & -e_1\tl d_2-\tl e_{13}\tl d_3 f_2& \tl e_{13}\tl d_3\\
-\tl d_2f_1-e_2\tl d_3 \tl f_{31} & \tl d_2+e_2\tl d_3f_1& -e_2\tl d_3\\
\tl d_3\tl f_{31} & -\tl d_3f_2 & \tl d_3
\end{bmatrix},
\end{split}
\end{align}
where $\tl e_{ij}$ and $\tl f_{ji}$ are defined in \eqref{eq:def-tilde-e-f}.
\begin{lem}\label{efij}
We have $e_{ij}(u)=f_{ji}(-u-i)$ for $1\lle i<j\lle 3$.
\end{lem}
\begin{proof}
The formulas for $j=i+1$ are a direct consequence of \eqref{eq:o2-2} and Corollary \ref{theta} using the maps $\vartheta_0$ and $\vartheta_1$. In particular, we have $e_2^{(1)}=-f_2^{(1)}$. Using Lemma \ref{lem:ei-generate-eij}, we have
\[
e_{13}(u)=[e_{1}(u),e_2^{(1)}]=[f_1(-u-1),-f_2^{(1)}]=[f_2^{(1)},f_1(-u-1)]=f_{31}(-u-1).\qedhere
\]
\end{proof}

\begin{lem}
We have
\begin{align}
(u-v)[e_1(u),e_2(v)]&= e_1(u)e_2(v)-e_1(v)e_2(v)-e_{13}(u)+e_{13}(v),\label{e1e2u}\\
[e_2^{(1)},\tl d_3(v)]&= -e_2(v)\tl d_3(v)-\tl d_3(v)f_2(v).\label{e1td2}
\end{align}
\end{lem}
\begin{proof}
Setting $i=1$, $j=k=2$, and $l=3$ in \eqref{sts}, we obtain that 
\[
(u-v)[s_{12}(u),\tl s_{23}(v)]=s_{11}(u)\tl s_{13}(v)+s_{12}(u)\tl s_{23}(v)+s_{13}(u)\tl s_{33}(v)
\]
which can be rewritten via \eqref{SS3} as 
\[
(u-v)[d_1(u)e_1(u),-e_2(v)\tl d_3(v)]=d_1(u)\big(-e_{13}(v)+e_1(v)e_2(v)-e_1(u)e_2(v)+e_{13}(u)\big)\tl d_3(v).
\]
Then the relation \eqref{e1e2u} follows by cancelling out $d_1(u)$ and $\tl d_3(v)$ since $\tl d_3(v)$ commutes with $d_1(u)$ and $e_1(u)$ while $d_1(u)$ commutes with $e_2(v)$; here we have applied Corollary \ref{cor:commu-d-e-f}.

The relation \eqref{e1td2} follows from \eqref{o2-6b} by taking the coefficients of $u^{-1}$ and applying the map $\vartheta_1$ in \eqref{eq:theta-m}; here we have applied Lemma \ref{lem:theta} and Corollary \ref{theta}.
\end{proof}

\begin{lem}
We  have
\beq\label{e13f2}
[e_{13}(u),e_2(v)]=e_2(v)[e_1(u),e_2(v)]+\frac{1}{u+v+1}\tl d_2(v)d_3(v)\big(f_1(v)-e_1(u)\big).
\eeq
\end{lem}
\begin{proof}
Applying \eqref{sts} with $i=1$, $j=k=3$, and $l=2$, one obtains
\[
(u-v)[s_{13}(u),\tl s_{32}(v)]=s_{11}(u)\tl s_{12}(v)+s_{12}(u)\tl s_{22}(v)+s_{13}(u)\tl s_{32}(v).
\]
Applying \eqref{SS3} and cancelling $d_1(u)$ from the left (thanks to the commutativity of $d_1(u)$ and $\tl s_{32}(v)$), we have
\beq\label{o3pf1}
\begin{split}
(u-v)[e_{13}(u),\tl d_3(v)f_2(v)]=&\,e_1(v)\tl d_2(v)+\tl e_{13}(v)\tl d_3(v)f_2(v)\\
-&\,e_1(u)\tl d_2(v)-e_1(u)e_2(v)\tl d_3(v)f_2(v)+e_{13}(u)\tl d_3(v)f_2(v).
\end{split}
\eeq
By \eqref{e1e2u}, the right-hand side of \eqref{o3pf1} is equal to
\begin{align*}
&\, \big(e_1(v)-e_1(u)\big)\tl d_2(v)+\big(\tl e_{13}(v)-e_1(u)e_2(v)+e_{13}(u)\big)\tl d_3(v)f_2(v)
\\=&\, \big(e_1(v)-e_1(u)\big)\tl d_2(v)+\big(e_1(v)e_2(v)-e_{13}(v)-e_1(u)e_2(v)+e_{13}(u)\big)\tl d_3(v)f_2(v)\\
=&\, \big(e_1(v)-e_1(u)\big)\tl d_2(v)-(u-v)[e_1(u),e_2(v)]\tl d_3(v)f_2(v).
\end{align*}
On the other hand, the left-hand side of \eqref{o3pf1} is equal to
\begin{align*}
(u-v)&[e_{13}(u),\tl d_3(v)]f_2(v)+(u-v)\tl d_3(v)[e_{13}(u),f_2(v)]\\
\stackrel{\eqref{eq:ei-generate-eij}}{=}\  &\ (u-v)\big[[e_1(u),e_2^{(1)}],\tl d_3(v)\big]f_2(v)+(u-v)\tl d_3(v)[e_{13}(u),f_2(v)]
\\= \ \ & \, \,   (u-v)\big[e_1(u),[e_2^{(1)},\tl d_3(v)]\big]f_2(v)+(u-v)\tl d_3(v)[e_{13}(u),f_2(v)]
\\ \stackrel{\eqref{e1td2}}{=}&\  (u-v)[e_1(u),-e_2(v)\tl d_3(v)-\tl d_3(v)f_2(v)]f_2(v)+(u-v)\tl d_3(v)[e_{13}(u),f_2(v)]\\
= \ \hspace{0.1cm}&\, -(u-v)[e_1(u),e_2(v)]\tl d_3(v)f_2(v)-(u-v)\tl d_3(v)[e_1(u),f_2(v)]f_2(v)\\
& \qquad\qquad\qquad \qquad\qquad\qquad \qquad\qquad\qquad\qquad+(u-v)\tl d_3(v)[e_{13}(u),f_2(v)].
\end{align*}
Therefore, 
\[
\big(e_1(v)-e_1(u)\big)\tl d_2(v)=-(u-v)\tl d_3(v)[e_1(u),f_2(v)]f_2(v)+(u-v)\tl d_3(v)[e_{13}(u),f_2(v)].
\]
Multiplying $d_3(v)$ from the left and noting that $d_3(v)$ commutes with $\big(e_1(v)-e_1(u)\big)$, we have
\beq\label{o3pf2}
[e_{13}(u),f_2(v)]=\frac{1}{u-v}\big(e_1(v)-e_1(u)\big)\tl d_2(v)d_3(v)+[e_1(u),f_2(v)]f_2(v).
\eeq
The desired relation \eqref{e13f2} now follows by applying the anti-automorphism $\eta$ to \eqref{o3pf2} with the substitution $u\mapsto -u-1$; here we have applied Lemma \ref{lem:tau} and Lemma \ref{efij}.
\end{proof}

\begin{lem}\label{e1e2e2lem}
We have
\beq\label{e1e2e2}
\big[[e_1(u),e_2(v)],e_2(v)\big]=\frac{\tl d_2(v)d_3(v)}{u-v-1}\Big(\frac{f_1(v)-e_1(v+1)}{2v+2}-\frac{f_1(v)-e_1(u)}{u+v+1}\Big).
\eeq
\end{lem}
\begin{proof}
We have
\begin{align*}
& (u-v)\big[[e_1(u),e_2(v)],e_2(v)\big]
\\
&\stackrel{\eqref{e1e2u}}{=}[e_1(u)e_2(v)-e_1(v)e_2(v)-e_{13}(u)+e_{13}(v),e_2(v)]\\
&\stackrel{\eqref{e13f2}}{=} [e_1(u),e_2(v)]e_2(v)-[e_1(v),e_2(v)]e_2(v)-e_2(v)[e_1(u),e_2(v)]+e_2(v)[e_1(v),e_2(v)]\\
&\qquad \qquad  -\frac{1}{u+v+1}\tl d_2(v)d_3(v)\big(f_1(v)-e_1(u)\big)+\frac{1}{2v+1}\tl d_2(v)d_3(v)\big(f_1(v)-e_1(v)\big)\\
&\;\; = \big[[e_1(u),e_2(v)],e_2(v)\big]- \big[[e_1(v),e_2(v)],e_2(v)\big]\\ &\qquad \qquad   -\frac{1}{u+v+1}\tl d_2(v)d_3(v)\big(f_1(v)-e_1(u)\big)+\frac{1}{2v+1}\tl d_2(v)d_3(v)\big(f_1(v)-e_1(v)\big).
\end{align*}
Therefore, 
\beq\label{o3pf3}
\begin{split}
-\big[[e_1(v),&\,e_2(v)], e_2(v)\big]+\frac{1}{2v+1}\tl d_2(v)d_3(v)\big(f_1(v)-e_1(v)\big)\\
=&\, (u-v-1) \big[[e_1(u),e_2(v)],e_2(v)\big]+\frac{1}{u+v+1}\tl d_2(v)d_3(v)\big(f_1(v)-e_1(u)\big).
\end{split}
\eeq
Note that the left-hand side of \eqref{o3pf3} is independent of $u$. Setting $u=v+1$, we find the left-hand side of \eqref{o3pf3} is equal to 
$\frac{1}{2v+2}\tl d_2(v)d_3(v)\big(f_1(v)-e_1(v+1)\big).$
Plugging this expression back into \eqref{o3pf3} proves the desired relation \eqref{e1e2e2}.
\end{proof}

Now we are ready to establish the Serre relations. Set
\begin{align*}
&\beta_1(u,v)=\frac{1}{u+v+1}\tl d_2(v)d_3(v)\big(f_1(v)-e_1(u)\big),\\
&\beta_2(u,v)=\frac{1}{u+v+2}\big(\tl d_2(u)d_3(u)-\tl d_2(v) d_3(v)\big),\\
&\beta_3(u,v)=\frac{\tl d_2(v)d_3(v)}{u-v-1}\Big(\frac{f_1(v)-e_1(v+1)}{2v+2}-\frac{f_1(v)-e_1(u)}{u+v+1}\Big).
\end{align*}
Then it follows from \eqref{e13f2}, \eqref{eq:o2-9}, and \eqref{e1e2e2}, respectively, that
\begin{align}
[e_{13}(u),e_2(v)]&=e_2(v)[e_1(u),e_2(v)]+\beta_1(u,v),\label{beta1}\\
[e_2(u),e_2(v)]&=\frac{1}{u-v}\big(e_2(u)-e_2(v)\big)^2+\beta_2(u,v),\label{beta2}\\
\big[[e_1(u),e_2(v)],e_2(v)\big]&=\beta_3(u,v).\label{beta3}
\end{align}
Comparing these relations with their counterparts of the Yangian $\rY(\gl_N)$ in \cite[\S5]{BK05}, one finds that the summands $\beta_i(u,v)$ are the extra terms appearing in the relations of twisted Yangians.

\begin{lem}\label{serrelem}
We have
\beq\label{serre}
\begin{split}
\big[[e_1(u),e_2(v)],e_2(w)\big]+\{v\leftrightarrow w\}&
\\
= \frac{1}{u-v}\Big( \beta_1(v,w)-\,\beta_1(u,w)&+\big(e_1(u)-e_1(v)\big)\beta_2(v,w)\\&+\beta_3(u,w)-\beta_3(v,w)\Big)+\{v\leftrightarrow w\}.
\end{split}
\eeq
\end{lem}
Here and below $\{v\leftrightarrow w\}$ denotes a summand obtained from the prior one with $v, w$ switched.

\begin{proof}
By \eqref{e1e2u}, we have
\begin{align*}
&(u-v)(u-w)(v-w)\big[[e_1(u),e_2(v)],e_2(w)\big]
\\
&= (u-w)(v-w)[e_1(u)e_2(v)-e_1(v)e_2(v)+e_{13}(v)-e_{13}(u),e_2(w)].
\end{align*}

Applying \eqref{beta1} and \eqref{beta3} we bring this to the form
\begin{align*}
=&\, (u-w)(v-w)[e_1(u),e_2(w)]e_2(v)+(u-w)(v-w)e_1(u)[e_2(v),e_2(w)]\\
&-(u-w)(v-w)[e_1(v),e_2(w)]e_2(v)-(u-w)(v-w)e_1(v)[e_2(v),e_2(w)]\\
&+(u-w)(v-w)\big([e_1(v),e_2(w)]e_2(w)+ \beta_1(v,w)-\beta_3(v,w) \big)\\
&-(u-w)(v-w)\big([e_1(u),e_2(w)]e_2(w)+ \beta_1(u,w)-\beta_3(u,w) \big)
\end{align*}
which is rewritten by \eqref{beta2} as
\begin{align*}
=&\,(u-w)(v-w)[e_1(u),e_2(w)]e_2(v)-(u-w)(v-w)[e_1(v),e_2(w)]e_2(v)\\
&+(u-w)e_1(u)\big((e_2(v)-e_2(w))^2+ (v-w)\beta_2(v,w) \big)\\
&-(u-w)e_1(v)\big((e_2(v)-e_2(w))^2+ (v-w)\beta_2(v,w) \big)\\
&+(u-w)(v-w)\big([e_1(v),e_2(w)]e_2(w)+ \beta_1(v,w)-\beta_3(v,w) \big)\\
&-(u-w)(v-w)\big([e_1(u),e_2(w)]e_2(w)+ \beta_1(u,w)-\beta_3(u,w) \big)
\\
= &\;\Xi(u,v,w) +\La(u,v,w).
\end{align*}
Here $\Xi(u,v,w)$ (and resp., $\La(u,v,w)$) denotes the sum of the summands above which do not involve (and resp., involve) $\beta$'s; that is,
\begin{align*}
\Xi(u,v,w)  
:=&\,(u-w)(v-w)[e_1(u),e_2(w)]e_2(v)-(u-w)(v-w)[e_1(v),e_2(w)]e_2(v)\\
&+(u-w)e_1(u) (e_2(v)-e_2(w))^2 
-(u-w)e_1(v) (e_2(v)-e_2(w))^2\\
&+(u-w)(v-w) [e_1(v),e_2(w)]e_2(w) -(u-w)(v-w) [e_1(u),e_2(w)]e_2(w),
\end{align*}
and 
\begin{align*}
\La(u,v,w)
:=&\, (u-w)(v-w) \Big(e_1(u) \beta_2(v,w)  
- e_1(v) \beta_2(v,w) \\
&\qquad\qquad\qquad\quad + \big(\beta_1(v,w)-\beta_3(v,w) \big)
- \big(\beta_1(u,w)-\beta_3(u,w) \big) \Big).
\end{align*}
%

Plugging \eqref{e1e2u} into $\Xi(u,v,w)$, one finds that $\Xi(u,v,w)$ is symmetric in $v$ and $w$; this is exactly the same as in the proof of \cite[Lemma 5.7]{BK05}. Therefore, we have
\begin{align*}
&(u-v)(u-w)(v-w)\big(\big[[e_1(u),e_2(v)],e_2(w)\big]+\{v\leftrightarrow w\}\big)\\
=&\, \Xi(u,v,w)+\La(u,v,w)-\Xi(u,w,v)-\La(u,w,v)
\\
=&\, \La(u,v,w)-\La(u,w,v),
\end{align*} 
proving \eqref{serre}.
\end{proof}

Recall from \eqref{bdef}--\eqref{hdef} the series
$b_i(u)=\sum_{r\in \bN} b_{i,r}u^{-r-1},\, h_{i}(u)=1+\sum_{r\in\bN}h_{i,r}u^{-r-1},$ for $i=1,2.$ Recall that 
$h_i(u)=h_i(-u)$,
i.e., $h_{i,2r}=0$, for all $r\in \bN$. 
Denote by $\scriptsize \begin{bmatrix} c_{11} & c_{12}\\ c_{21} & c_{22} \end{bmatrix} =\begin{bmatrix} 2 & -1\\ -1 & 2\end{bmatrix}$ the Cartan matrix of type $A_2$. 

\begin{thm}\label{so3rel}
The algebra $\SX_3$ is generated by the coefficients of $b_i(u)$ and $h_i(u)$, $i=1,2$. Moreover, the following relations hold in $\SX_3$, for $i,j\in \{1,2\}$, $k_1, k_2,r, s\in\bN$:
\begin{align}
[h_{i,r},h_{j,s}] &=0,  \qquad h_{i,2r}=0,\label{o3hh}\\
[h_{i,r+1},b_{j,s}]-[h_{i,r-1},b_{j,s+2}] &=c_{ij}\{h_{i,r-1},b_{j,s+1}\}+\frac{1}{4}c_{ij}^2[h_{i,r-1},b_{j,s}],\label{o3hb}\\
[b_{i,r+1},b_{j,s}]-[b_{i,r},b_{j,s+1}] &=\frac{c_{ij}}{2}\{b_{i,r},b_{j,s}\} -2\delta_{ij}(-1)^rh_{i,r+s+1},\label{o3bb}\\
\mathrm{Sym}_{k_1,k_2}\big[b_{i,k_1},[b_{i,k_2},b_{j,r}] \big]& \notag \\
=
(-1)^{k_1}\sum_{p\gge 0}2^{-2p} & \big([h_{i,k_1+k_2-2p-1},b_{j,r+1}]-\{h_{i,k_1+k_2-2p-1},b_{j,r}\}\big),
\qquad \text{ if } i\neq j.\label{o3serre}
\end{align}
(It is understood that $h_{i,-1}=1$ above.)
\end{thm}
It will be proved later in greater generality (see Theorem \ref{mainthm2}) that the relations \eqref{o3hh}--\eqref{o3serre} are defining relations for $\SX_3$. 

\begin{proof}
The relations for $i=j$ are clear by Proposition \ref{sx2p} and Corollary \ref{theta2}. For $i\ne j$, the relation \eqref{o3bb}  follows from \eqref{e1e2u} while \eqref{o3hh} follows by Corollary \ref{cor:commu-d-e-f}. It remains to verify \eqref{o3hb} for $i\neq j$ and the Serre relation \eqref{o3serre}. 

Let us prove first \eqref{o3hb} for $i=2$ and $j=1$. It follows from \eqref{eq:o2-7} and $[d_3(u),b_1(v)]=0$ that
\[
[h_2(u),b_1(v)]=\frac{1}{u-v-\frac12} h_2(u)\big(b_1(u-\tfrac12)-b_1(v)\big)+\frac{1}{u+v-\frac12} \big(b_1(v)-b_1(-u+\tfrac12)\big)h_2(u).
\]
Clearing the denominators and simplifying further, we have
\[
(u^2-v^2)[h_2(u),b_1(v)]=-v\big(h_2(u)b_1(v)+b_1(v)h_2(u)\big)+\frac{1}{4}[h_2(u),b_1(v)]+\cdots,
\]
where $\cdots$ stands for terms that will not contribute to the coefficients of $u^{-r-1}v^{-s-1}$ for $r,s\in\bN$. The relation \eqref{o3hb} for $i=2$ and $j=1$ follows. 

By Lemma \ref{zetamap}, the automorphism $\zeta_3:\X_3\to \X_3$ maps
$b_i(u) \mapsto -b_{3-i}(u)$ and $h_i(u) \mapsto h_{3-i}(u).$
Applying $\zeta_3$ to \eqref{o3hb} for $i=2$ and $j=1$, we obtain \eqref{o3hb} for $i=1$ and $j=2$.

The proof of the Serre relation \eqref{o3serre} is long and will be given in  \S\ref{subsec:Serre} below.
\end{proof}

\begin{lem}\label{serresimple}
For $i\neq j \in \{1,2\}$, $m\gge 0$, and $r\gge 1$, we have
\[
[h_{i,m+1},b_{j,r-1}]=\sum_{p\gge 0}2^{-2p}\big([h_{i,m-2p-1},b_{j,r+1}]-\{h_{i,m-2p-1},b_{j,r}\}\big).
\]
\end{lem}

\begin{proof}
By \eqref{o3hb}, we have
\begin{align*}
[h_{i,m-2p-1},b_{j,r+1}]- \{h_{i,m-2p-1},b_{j,r}\}
= [h_{i,m-2p+1},b_{j,r-1}]-\tfrac{1}{4}[h_{i,m-2p-1},b_{j,r-1}].
\end{align*}
The lemma follows from this and the fact that $h_{i,-1}=1$.
\end{proof}

Lemma~\ref{serresimple} allows to simplify the Serre relation \eqref{o3serre}, for $r\gge 1$ and $i\neq j$ to be
\begin{align}\label{eq:simpleSerre}
    \mathrm{Sym}_{k_1,k_2}\big[b_{i,k_1},[b_{i,k_2},b_{j,r}] \big] &=(-1)^{k_1}[h_{i,k_1+k_2+1},b_{j,r-1}], \qquad \text{ for } r\gge 1,\, i\neq j.
\end{align}

\begin{rem}
One can reformulate the relations \eqref{o3hh}--\eqref{o3bb} in generating series as follows: 
\begin{align*}
[h_i(u),h_j(v)]&=0,
\\
 (u^2-v^2)[h_i(u),b_j(v)]&=c_{ij}v\{h_i(u),b_j(v)\}+\frac{1}{4}c_{ij}^2[h_i(u),b_j(v)]\\&\quad -[h_i(u),b_{j,1}]-c_{ij}\{h_i(u),b_{j,0}\}-v[h_i(u),b_{j,0}], 
\\
(u-v)[b_i(u),b_i(v)]&=-\frac{c_{ii}}{2}\big(b_i(u)-b_i(v)\big)^2- \frac{u-v}{u+v}\big(h_i(u)-h_i(v)\big),
\\
(u-v)[b_i(u),b_j(v)]&=\frac{c_{ij}}{2}\{b_i(u),b_j(v)\}+\big([b_{i,0},b_j(v)]-[b_i(u),b_{j,0}]\big).
\end{align*}
The relation \eqref{eq:simpleSerre} leads to the following formulation of the Serre relation in generating series:
\begin{align*}
 &\mathrm{Sym}_{u,v}\big[b_{i}(u),[b_{i}(v),b_{j}(w)] \big]
 \\
 =&-\frac{w^{-1}}{1+uv^{-1}}[h_{i}(u),b_{j}(w)]
+ w^{-1}\mathrm{Sym}_{u,v}\big[b_{i}(u),[b_{i}(v),b_{j,0}] \big].
\end{align*}
\end{rem}

\subsection{Proof of the Serre relations } \label{subsec:Serre}

Now we prove the Serre relations \eqref{o3serre}. Our strategy bears similarities with \cite{Le93} while the details here are more involved. Denote  the left and right-hand sides  of \eqref{o3serre} (cf. Lemma \ref{serresimple}) by
\begin{align}
\mathbb S_{ij}(k_1,k_2;r) &=\mathrm{Sym}_{k_1,k_2}\big[b_{i,k_1},[b_{i,k_2},b_{j,r}]\big],
\\
\mathbb F_{ij}(k_1,k_2;r) &=(-1)^{k_1}\sum_{p\gge 0}2^{-2p}\big([h_{i,k_1+k_2-2p-1},b_{j,r+1}]-\{h_{i,k_1+k_2-2p-1},b_{j,r}\}\big).
\end{align} 

\medskip 

{\bf Claim 1}. We have $\mathbb S_{ij}(k,0;0)=\mathbb F_{ij}(k,0;0)$.

\begin{proof}
 Let us compute $\mathbb S_{21}(k,0;0)$ first when $k$ is even. We do the following 3 transformations.
\begin{itemize}
    \item Apply the anti-automorphism $\eta$ to \eqref{serre} and obtain such a relation between generating series  $f_i(z), d_j(z)$.
    \item To get the Serre relations in terms of $h_i(z),b_i(z)$, we shift the spectral parameters properly using \eqref{bdef} and \eqref{hdef}, namely $u\to u-\frac{1}{2}$, $v\to v-1$, and $w\to w-1$.
    \item Finally, we expand the series in the region $\{(u,v,w):v\gg w\gg u\}$ and compare the coefficients of $u^{-1}v^{-1}w^{-k-1}$.
\end{itemize}
The expression $\mathbb S_{21}(k,0;0)$ is equal to the coefficients of $u^{-1}v^{-1}w^{-k-1}$ in the left-hand side of \eqref{serre} and hence is the same as the coefficients of $u^{-1}v^{-1}w^{-k-1}$ in this expansion from the right-hand side of \eqref{serre}. Due to the suitable choice of the region, it turns out the nontrivial contribution of the term $u^{-1}v^{-1}w^{-k-1}$ from the right-hand side only comes from the following terms (at this stage we only need terms containing $b_1(u)$ and $v^{-j}$ for $j=0,1$),
\begin{align*}
\frac{1}{(u-v+\frac{1}{2})(u+w-\frac{1}{2})}b_1(u)h_2(w)+\frac{1}{(u-v+\frac{1}{2})(u-w-\frac{1}{2})(u+w-\frac{1}{2})}b_1(u)h_2(w)&\\
+\frac{1}{(u-w+\frac{1}{2})(v+w)}(h_2(w)-h_2(v))b_1(u)+\frac{1}{(u-w+\frac{1}{2})(u+v-\frac{1}{2})}b_1(u)h_2(v)&.
\end{align*}
Expanding it at $v=\infty$ and taking the coefficients of $v^{-1}$, we obtain
\be
-\frac{1}{w+u-\frac{1}{2}}b_1(u)h_2(w)+\frac{1}{w^2-(u-\frac{1}{2})^2}b_1(u)h_2(w)-\frac{1}{w-u-\frac{1}{2}}h_2(w)b_1(u) =: \Omega(u,w).
\ee 
Since $h_2(w)=h_2(-w)$ and $k$ is even, the coefficient of $u^{-1}w^{-k-1}$ in
$
\frac{1}{w^2-(u-\frac{1}{2})^2}b_1(u)h_2(w)$
is zero. 

We shall write 
\[
A(u,w)\simeq B(u,w)
\]
if $A(u,w)$ and $B(u,w)$ have the same coefficients of $u^{-k}w^{-l}$, for $k,l>0$. In such a notation we have
\[
[h_2(w),b_1(u)]\simeq -\frac{1}{w-u-\frac{1}{2}}h_2(w)b_1(u)+\frac{1}{w+u-\frac{1}{2}}b_1(u)h_2(w).
\]
This implies further that
\beq\label{o3pf12}
\frac{w-u+\frac{1}{2}}{w-u-\frac{1}{2}}h_2(w)b_1(u)\simeq \frac{w+u+\frac{1}{2}}{w+u-\frac{1}{2}}b_1(u)h_2(w).
\eeq
From this one obtains that
\be
-\frac{1}{w+u-\frac{1}{2}}b_1(u)h_2(w)-\frac{1}{w-u-\frac{1}{2}}h_2(w)b_1(u)\simeq \frac{u}{w-\frac12}[h_2(w),b_1(u)]-\frac{1}{w-\frac12}\{h_2(w),b_1(u)\}.
\ee
The coefficient of $u^{-1}w^{-k}$ in the right-hand side is equal to
\[
\sum_{p\gge 0}2^{-2p}\Big([h_{2,k-2p-1},b_{1,1}]-\{h_{2,k-2p-1},b_{1,0}\}\Big).
\]
Thus we have
\[
\mathbb S_{21}(k,0;0)=\sum_{p\gge 0}2^{-2p}\big([h_{2,k-2p-1},b_{1,1}]-\{h_{2,k-2p-1},b_{1,0}\}\big).
\]

By a similar computation using \eqref{o3pf12} we also have that $\Omega(u,w)\simeq \Omega(u,-w)$. Thus, if $k$ is odd, then $\mathbb S_{21}(k,0;0)=0$.

Again, applying $\zeta_3$ to the above Serre relation, we also obtain another Serre relation
\[
\mathbb S_{12}(k,0;0)=\sum_{p\gge 0}2^{-2p}\big([h_{1,k-2p-1},b_{2,1}]-\{h_{1,k-2p-1},b_{2,0}\}\big),
\]
whenever $k$ is even or odd since $h_{i,2r}=0$.
\end{proof}

{\bf Claim 2}. We have $\mathbb S_{ij}(k,0;r)=\mathbb F_{ij}(k,0;r)$. Moreover, $\mathbb S_{ij}(k+1,0;r)+\mathbb S_{ij}(k,1;r)=0$.

\begin{proof}
We prove $\mathbb S_{ij}(k,0;r)=\mathbb F_{ij}(k,0;r)$ by induction on $r$. The base case is shown in Claim 1. Suppose $\mathbb S_{ij}(k,0;r)=\mathbb F_{ij}(k,0;r)$. By \eqref{o3hb} with $r=0$, we have 
\[
[h_{i,1},b_{j,s}]=2c_{ij}b_{j,s+1}.
\]

Applying $h_{i,1}$ and $h_{j,1}$ to $\mathbb S_{ij}(k,0;r)=\mathbb F_{ij}(k,0;r)$, we obtain
\begin{align*}
&4\bS_{ij}(k+1,0;r)+4\bS_{ij}(k,1;r)-2\bS_{ij}(k,0;r+1)=-2\bF_{ij}(k,0;r+1),\\
-&2\bS_{ij}(k+1,0;r)-2\bS_{ij}(k,1;r)+4\bS_{ij}(k,0;r+1)=4\bF_{ij}(k,0;r+1),
\end{align*}
respectively. Solving the system, we obtain that
\[
\bS_{ij}(k+1,0;r)+\bS_{ij}(k,1;r)=0,\qquad 
\bS_{ij}(k,0;r+1)=\bF_{ij}(k,0;r+1).
\]
Hence by induction, we have $\mathbb S_{ij}(k,0;r)=\mathbb F_{ij}(k,0;r)$ for all $r\in\bN$. In particular, from the proof, we also have $\bS_{ij}(k+1,0;r)+\bS_{ij}(k,1;r)=0$ for all $r\in\bN$.
\end{proof}

{\bf Claim 3}. We have $\mathbb S_{ij}(k_1+1,k_2;r)+\mathbb S_{ij}(k_1,k_2+1;r)=0$ and $\mathbb S_{ij}(k_1,k_2;r)=\mathbb F_{ij}(k_1,k_2;r)$. 
\begin{proof}
We prove it by induction on $k_1+k_2$. When $k_1+k_2=0$, this is clear as by definition $\bS_{ij}(k_1,k_2;r)$ is symmetric in $k_1,k_2$ and also $\bS_{ij}(1,0;r)=0$ by Claim 2.

Suppose $\mathbb S_{ij}(k_1+1,k_2;r)+\mathbb S_{ij}(k_1,k_2+1;r)=0$. Applying $h_{i,1}$ and $h_{j,1}$ to it as in the proof of Claim 2, we conclude that
\[
\bS_{ij}(k_1+2,k_2;r)+2\bS_{ij}(k_1+1,k_2+1;r)+\bS_{ij}(k_1,k_2+2;r)=0.
\]
Thus 
\begin{align*}
\bS_{ij}(k_1+1,k_2+1;r)+&\,\bS_{ij}(k_1,k_2+2;r)\\=&\,(-1)^{k_2+1}\big(\bS_{ij}(k_1+k_2+2,0;r)+\bS_{ij}(k_1+k_2+1,1;r)\big)=0
\end{align*}
by Claim 2. Therefore, $\mathbb S_{ij}(k_1+1,k_2;r)+\mathbb S_{ij}(k_1,k_2+1;r)=0$ for all $k_1,k_2,r\in\bN$.

It is clear that $\mathbb F_{ij}(k_1+1,k_2;r)+\mathbb F_{ij}(k_1,k_2+1;r)=0$ for all $k_1,k_2,r\in\bN$. Hence $\mathbb S_{ij}(k_1,k_2;r)=\mathbb F_{ij}(k_1,k_2;r)$ follows immediately from this property and Claim 2.
\end{proof}

\section{Drinfeld type presentations of twisted Yangians}
\label{sec:Drinfeld}

In this section we formulate and establish the Drinfeld type current presentations of (special) twisted Yangians $\Y_N$ and $\SY_N$. 

\subsection{Main results}

\begin{thm}  [Drinfeld presentation of the twisted Yangian $\Y_N$]
\label{mainthm}
Let $A=(c_{ij})_{1\lle i,j\lle N-1}$ be the Cartan matrix of type $A_{N-1}$, and set $h_{i,-1}=1$, $c_{0j}=-\delta_{j,1}$. Then the twisted Yangian $\Y_N$ is generated by $h_{i,r}$, $ b_{j,r}$, for $0\lle i<N$, $1\lle j<N$ and $r\in\bN$ subject only to the following relations, for $r,s \in \bN$:
\begin{align}
[ h_{i,r}, h_{j,s}]&=0, \qquad { h_{i,2r}=0, } \label{hhN}\\	
[ h_{i,r+1}, b_{j,s}]-[ h_{i,r-1}, b_{j,s+2}]&=c_{ij}\{ h_{i,r-1}, b_{j,s+1}\}+\frac{1}{4}c_{ij}^2[ h_{i,r-1}, b_{j,s}],\label{hbN}\\
[ b_{i,r+1}, b_{j,s}]-[ b_{i,r}, b_{j,s+1}]&=\frac{c_{ij}}{2}\{ b_{i,r}, b_{j,s}\}-2\delta_{ij}(-1)^r h_{i,r+s+1},\label{bbN}\\
[b_{i,r},b_{j,s}]&=0,\qquad  \text{ for }|i-j|>1,\label{bbN2}\\
\mathrm{Sym}_{k_1,k_2}\big[b_{i,k_1},[b_{i,k_2},b_{j,r}] \big] &=(-1)^{k_1}[h_{i,k_1+k_2+1},b_{j,r-1}],
\quad \text{ for } c_{ij}=-1.\label{serreN}
\end{align}
\end{thm}

Theorem \ref{mainthm} will be proved in \S\ref{subsec:proofmain} below. Note that the relation \eqref{hbN} with $r=0$ reads simply as
\beq\label{h1bN}
[h_{i,1},b_{j,s}]=2c_{ij}b_{j,s+1}.
\eeq

\begin{rem} \label{rem:Serre-1}
The right-hand side of the Serre relation \eqref{serreN} for $r=0$ is understood by formally validating \eqref{hbN} for $c_{ij}=-1\&s=-1$ or equivalently Lemma \ref{serresimple} for $r=0$. More explicitly, $[h_{i,m},b_{j,-1}]$ is defined inductively on $m$ by declaring that $[h_{i,-1},b_{j,-1}]=0$ and
$$
[h_{i,m+1},b_{j,-1}]=[h_{i,m-1},b_{j,1}]-\{h_{i,m-1},b_{j,0}\}+\frac{1}{4}[h_{i,m-1},b_{j,-1}],
$$
or equivalently,
\[
[h_{i,m+1},b_{j,-1}]=\sum_{p\gge 0}2^{-2p}  \Big([h_{i,m-2p-1},b_{j,r+1}]-\{h_{i,m-2p-1},b_{j,r}\}\Big).
\]
\end{rem}

A variant of Theorem \ref{mainthm} gives us a presentation for $\SY_N$ below. Note that $h_{0,r}$  $(r\in \bN)$ are additional generators in $\Y_N$, which are not present in $\SY_N$. 

\begin{thm}
[Drinfeld presentation of the special twisted Yangian $\SY_N$]
\label{mainthm2}
The special twisted Yangian $\SY_N$ is generated by $h_{i,r}$, $b_{i,r}$, for $1\lle i<N$ and $r\in\bN$ subject only to the same relations \eqref{hhN}--\eqref{serreN}, for $1\lle i,j<N$. Moreover, we have $\SX_N\cong \SY_N$.
\end{thm}

\begin{proof} 
Denote by $\mathscr S\YDrN$ the algebra generated by the above generators and relations. By the proof of Theorem \ref{mainthm}, it is clear that the elements $h_{i,r}$ and $b_{i,r}$ in $\SX_N$ defined in terms of Gauss decomposition satisfy the relations in Theorem \ref{mainthm} with admissible indices. Thus we have a surjective homomorphism $\mathscr S\YDrN \twoheadrightarrow \SX_N$. Then we have the following chain of algebra homomorphisms,
\[
\mathscr S\YDrN \twoheadrightarrow \SX_N \hookrightarrow \X_N \twoheadrightarrow \Y_N \twoheadrightarrow \SY_N,
\]
see Definitions \ref{tydef}, \ref{sxdef} and Lemma \ref{syquo}.

By Lemma \ref{lem:special-in-Gauss} and Definition \ref{sxdef}, the composition map $\SX_N \hookrightarrow \X_N \twoheadrightarrow \Y_N \twoheadrightarrow \SY_N$ gives us an epimorphism $\SX_N \twoheadrightarrow \SY_N$. In particular, we have the following surjective homomorphisms
\beq\label{pfmain1}
\psi:\mathscr S\YDrN \twoheadrightarrow \SX_N  \twoheadrightarrow \SY_N.
\eeq
Therefore, in order to show $\mathscr S\YDrN \cong \SX_N  \cong \SY_N$, it suffices to show that the homomorphism $\psi:\mathscr S\YDrN\to\SY_N$ in \eqref{pfmain1} is also injective. We achieve that by showing a spanning set of $\mathscr S\YDrN$ under the map $\psi$ is a basis of $\SY_N$.

It is known from Proposition \ref{prop:PBW} that there is a PBW basis of $\SY_N$ consisting of monomials in $h_{i,2r+1},b_{ji,r}$, $1\lle i<j\lle N$, $r\gge 0$, with certain (maybe any) given order. In order to show that $\mathscr S\YDrN \cong \SX_N  \cong \SY_N$, it suffices to show that there is a spanning set of $\mathscr S\YDrN$ that are sent to those monomials in $\SY_N$ under the homomorphism $\psi$. This is basically the same as done in the proof of Theorem \ref{mainthm}. 
\end{proof}

\begin{rem}
The Drinfeld presentation in Theorem \ref{mainthm} of the special twisted Yangians of type AI obtained via Gauss decomposition matches perfectly with the presentation obtained via degeneration of affine $\imath$quantum group in \cite{LWZ24}.
\end{rem}

\subsection{A flat deformation}

There exists a flat deformation $\rY(\gl_N,\hbar)$ of the Yangian $\rY(\gl_N)$, cf. \cite[Rem. 1.4.4\&2.4.5]{Mol07}, and this induces a flat deformation on $\Y_N$ as a subalgebra. Let us make this explicit using the new current generators. 

For each $\hbar\in\bC^\times$, consider the algebra $\Y_{N,\hbar}$ with generators $s_{ij}^{(r)}$ and the relations
\begin{align*}
(u^2-v^2)[s_{ij}(u),s_{kl}(v)]&=\hbar (u+v)\big(s_{kj}(u)s_{il}(v)-s_{kj}(v)s_{il}(u)\big)\\
&-\hbar(u-v)\big(s_{ik}(u)s_{jl}(v)-s_{ki}(v)s_{lj}(u)\big)\\
&\hskip1.02cm +\hbar^2\big(s_{ki}(u)s_{jl}(v)-s_{ki}(v)s_{jl}(u)\big)
\end{align*}
and
\[
s_{ji}(-u)=s_{ij}(u)+\hbar\frac{s_{ij}(u)-s_{ij}(-u)}{u},
\]
where
\[
s_{ij}(u)=\delta_{ij}+\hbar\sum_{r>0}s_{ij}^{(r)}u^{-r}.
\]

Similarly, applying the Gauss decomposition to the matrix $S(u)=(s_{ij}(u))$ satisfying \eqref{eq:sij-Gauss}, we again obtain generating series $e_{ij}(u)$, $f_{ji}(u)$, and $d_i(u)$. Set
\[
e_{ij}(u)=\hbar\sum_{r\gge 1}e_{ij}^{(r)}u^{-r},\quad f_{ji}(u)=\hbar\sum_{r\gge 1}f_{ji}^{(r)}u^{-r},\quad d_k(u)=1+\hbar\sum_{r\gge 1}d_{k}^{(r)}u^{-r}.
\]
We further set
\begin{align*}
b_i(u)&=\hbar\sum_{r\gge 0}b_{i,r}u^{-r-1}:= f_{i+1,i}\big(u-\tfrac{i\hbar}{2}\big),\\
h_0(u)&=1+\hbar\sum_{r\gge 0}h_{0,r}u^{-r-1}:=d_1(u),\\
h_i(u)&=1+\hbar\sum_{r\gge 0}h_{i,r}u^{-r-1}:=\tl d_{i}\big(u-\tfrac{i\hbar}{2}\big)d_{i+1}\big(u-\tfrac{i\hbar}{2}\big). 
\end{align*}
Then the relations in $\Y_{N,\hbar}$ become
\begin{align*}
[ h_{i,r}, h_{j,s}]&=0, \qquad { h_{i,2r}=0, } \\	
[ h_{i,r+1}, b_{j,s}]-[ h_{i,r-1}, b_{j,s+2}]&=c_{ij}\hbar \{ h_{i,r-1}, b_{j,s+1}\}+\frac{1}{4}c_{ij}^2\hbar^2[ h_{i,r-1}, b_{j,s}],\\
[ b_{i,r+1}, b_{j,s}]-[ b_{i,r}, b_{j,s+1}]&=\frac{c_{ij}\hbar}{2}\{ b_{i,r}, b_{j,s}\}-2\delta_{ij}(-1)^r h_{i,r+s+1},\\
[b_{i,r},b_{j,s}]&=0,\qquad  \text{ for }|i-j|>1,\\
\mathrm{Sym}_{k_1,k_2}\big[b_{i,k_1},[b_{i,k_2},b_{j,r}] \big]& =[h_{i,k_1+k_2+1},b_{j,r-1}],
\quad \text{ for } c_{ij}=-1.
\end{align*}
Here we have adopted the convention that $h_{i,-1}=\hbar^{-1}$. Note the novel feature in the above relations where an $\hbar^2$-term appears; compare \cite[Rem. 1.4.4\&2.4.5]{Mol07}. 

The algebras $\Y_{N,\hbar}$ with $\hbar\ne 0$ are all isomorphic to each other, with an explicit isomorphism $\Y_{N,\hbar}\to \Y_{N}=\Y_{N,1}$ given by $s_{ij}^{(r)}\mapsto s_{ij}^{(r)}\hbar^{r-1}$. The isomorphism in terms of the current generators is given by
\[
h_{i,r}\mapsto h_{i,r}\hbar^r,\qquad b_{j,s}\mapsto b_{j,s}\hbar^{s}.
\]

The relations above are homogeneous if we regard $\hbar$ as a formal parameter of degree $1$ and let $\deg h_{i,r} =\deg b_{i,r}=r$, for all $i,r$.

\subsection{Generalizations}
\label{subsec:gen}

Let $A=(c_{ij})$ be a Cartan matrix of type ADE. Then the generators and relations in Theorem~\ref{mainthm2} still make sense and define an algebra denoted by $\Y(A)$. Theorem~\ref{mainthm2} asserts that $\Y(A)$ associated to the type A$_{N-1}$  Cartan matrix is isomorphic to the special twisted Yangian associated to the orthogonal Lie algebra (or better, associated to the symmetric pair of type AI). 

We conjecture that $\Y(A)$ associated to the type D Cartan matrix is isomorphic to the twisted Yangians associated to symmetric pairs of split type DI introduced in \cite{GR16}. 

Associated to Cartan matrices $A$ of exceptional types $E_6$-$E_7$-$E_8$, the algebras $\Y(A)$ are new. They should be viewed as twisted Yangians associated to symmetric pairs of split type EI, EV and EVIII; they are deformations of twisted current algebras $\g [z]^\theta$ associated to the Chevalley involution $\omega$ of a simple Lie algebra $\g$ of type E, where $\theta (x\otimes z^r) =\omega(x) \otimes (-z)^r$. 

In a totally different (degeneration) approach, we have constructed in \cite{LWZ24} a class of twisted Yangian type algebras from Drinfeld presentations \cite{LW21, Z22} of split affine $\imath$quantum groups. These algebras from degeneration will be shown to coincide with $\Y(A)$ associated to Cartan matrix $A$ of type ADE here, and the algebra $\Y(A)$ will be shown to admit a PBW basis of the expected size.

For a Cartan matrix $A=(c_{ij})$ of {\it affine} ADE type (excluding A$_1^{(1)}$), the algebra $\Y(A)$ defined by the generators and relations in Theorem~\ref{mainthm2} remains valid. We shall refer to this algebra as the {\em twisted affine Yangians} (of split affine ADE types). They are deformations of universal central extensions of twisted double current algebras $\g [t^{\pm 1}, z]^\theta$ associated to the Chevalley involution $\omega$ on $\g$, where $\theta (x\otimes t^s z^r) =\omega(x) \otimes t^{-s} (-z)^r$. These algebras should be viewed as a twisted analogue of affine Yangians. One expects that they admit a degeneration to a twisted analogue of deformed double current algebras (cf. \cite{G07}). 

The construction in this paper opens the door for constructing the Drinfeld type current presentations for the twisted Yangians associated to symmetric pairs of type ABCD \cite{Ol92,MR02,GR16} and also for defining new twisted Yangians associated to symmetric pairs of type FI and GI in current presentations; compare \cite{LWZ24}. 

Brundan-Kleshchev's parabolic presentations of the Yangian $\rY(\gl_N)$ \cite{BK05} (generalizing Drinfeld's current presentation) allowed them to further develop connections to finite W-algebras of type A through a notion of shifted Yangians. The current presentation of the twisted Yangian $\Y_N$ in this paper leads to parabolic presentations of $\Y_N$, shifted twisted Yangians, and applications to finite W-algebras of classical type (cf. \cite{Br16}). This has been addressed in \cite{LPT+25}.

\subsection{Proof of the Drinfeld type presentation}
\label{subsec:proofmain}

\begin{proof}[Proof of Theorem \ref{mainthm}]
We first show that the twisted Yangian $\Y_N$ has generators $h_{i,r}$, $ b_{j,r}$ and these generators satisfy the relations stated in the theorem. Define $h_{i,r},b_{i,r}$ via Gauss decomposition as in \eqref{hdef}, \eqref{bdef} for $1\lle i<N$ and $r\in\bN$. Define $h_0(u)=\sum_{r\gge 0}h_{0,r}u^{-r-1}=d_1(u)$, namely $h_{0,r}=d_1^{(r+1)}$. It is clear from the proof of Lemma \ref{lem:special-in-Gauss}, cf. also Lemma \ref{lem:ei-generate-eij}, that $h_{i,r},b_{j,r}$ for $0\lle i<N$, $1\lle j<N$, $r\in\bN$, generate the twisted Yangian $\Y_N$.

We prove that $h_{i,r},b_{i,r}$ satisfy the listed relations by induction on $N$. The base cases $N=2,3$ are essentially confirmed in Proposition \ref{sx2p} and Theorem \ref{so3rel}, respectively.  The condition $h_{0,2r}=0$ follows from Lemma \ref{lem:s11} as $h_0(u)=d_1(u)=s_{11}(u)$ is even. Suppose now the relations hold for the twisted Yangian $\Y_{N}$. We would like to prove it for the  $\Y_{N+1}$ case. Recall that we have the natural embedding $\Y_N\to \Y_{N+1}$, $s_{ij}(u)\to s_{ij}(u)$ which sends $h_{i}(u)$, $b_j(u)$ of $\Y_N$ to the series in $ \Y_{N+1}$ of the same name. Note that we also have the homomorphism $\vartheta_1:\X_N\to \X_{N+1}$ with the properties in Corollaries \ref{theta} and \ref{theta2}. Hence the above relations hold for the case when $i,j<N$ or $i,j\gge 2$. It remains to show the relations for the case when ($i\in\{0,1\}$ and $j=N$) or ($i=N$ and $j\in\{0,1\}$). This follows immediately from Corollary \ref{cor:commu-d-e-f} when $N\gge 3$.

Denote by $\YDrN$ the algebra generated by the above generators and relations as in the statement of the theorem. The above argument implies that there is a surjective homomorphism
\beq\label{surjm}
\YDrN \twoheadrightarrow \Y_N,
\eeq
which sends the generators $h_{i,r}$ $b_{j,r}$ of $\YDrN$ to the elements of $\Y_N$ denoted by the same symbols. Hence to complete the proof, we need to prove the homomorphism is injective. By Proposition \ref{PBWgauss}, the set of monomials in 
\[
h_{i,2r+1}\qquad \text{with}\qquad 0\lle i< N,\quad r\gge 0
\]
and
\[
b_{ij,r} \qquad \text{with}\qquad 1\lle j<i\lle N,\quad r\gge 0,
\]
taken in some fixed order is linearly independent in the twisted Yangian $\Y_N$.

For any $1\lle j<i\lle N$, define elements $\ttb_{ij}^{(r)}$ in $\YDrN$ inductively by the rule:
\beq\label{pfN0}
\ttb_{i+1,i}^{(0)}=b_{i,0},\qquad \ttb_{i+1,j}^{(r)}=[\ttb_{i+1,i}^{(0)},\ttb_{ij}^{(r)}]
\eeq
for $j<i$. Note that this definition is consistent with \eqref{eq:ei-generate-eij}. Define 
\beq\label{pfN1}
\ttb_{ii}^{(r)}=h_{0,r}+h_{1,r}+\cdots+h_{i-1,r}
\eeq
and set $\ttb_{ji}^{(r)}=(-1)^{r+1}\ttb_{ij}^{(r)}$ for $i>j$. Then we have $\ttb_{ii}^{(r)}+(-1)^r\ttb_{ii}^{(r)}=0$, namely $\ttb_{ii}^{(r)}=0$ if $r$ is even. Clearly, the homomorphism \eqref{surjm} maps $\ttb_{ij}^{(r)}$ (with $i>j$) and $\ttb_{ii}^{(r)}$ of $\YDrN$ to $b_{ij,r}$ and $h_{0,r}+h_{1,r}+\cdots+h_{i-1,r}$ of $\Y_N$. Theorem \ref{mainthm} will follow if we prove that the algebra $\YDrN$ is spanned by the monomials in $\ttb_{ij}^{(r)}$, $1\lle j<i\lle N$, and $h_{i,2r+1}$, $1\lle i\lle N$, $r\in\bN$, taken in some fixed order.

Define an ascending filtration on $\YDrN$ by setting $\deg b_{i,r}=\deg h_{i,r}=r$. Let $\gr\,\YDrN$ be the associated graded algebra. Let $\bar \ttb_{ij}^{(r)}$, $\bar h_{i,r}$ be the images of $\ttb_{ij}^{(r)}$, $h_{i,r}$, respectively, in the $r$-th component of the graded algebra $\gr\,\YDrN$. We claim that these images satisfy
\beq\label{todo}
[\bar \ttb_{ij}^{(r)},\bar\ttb_{kl}^{(s)}]=\delta_{jk}\bar\ttb_{il}^{(r+s)}-\delta_{li}\bar\ttb_{kj}^{(r+s)}-(-1)^r\delta_{ki}\bar\ttb_{jl}^{(r+s)}+(-1)^r\delta_{lj}\bar\ttb_{ki}^{(r+s)}.
\eeq
By the relation $\bar \ttb_{ji}^{(r)}=(-1)^{r+1}\bar \ttb_{ij}^{(r)}$, it suffices to show it for the cases when $i\gge j$ and $k\gge l$.

We prepare some relations that will be used later. Observe from \eqref{hhN}, \eqref{hbN} and \eqref{h1bN} that
\beq\label{pfN3}
[\bar h_{i,r},\bar\ttb_{k+1,k}^{(s)}]=2(1-\delta_{r,ev})c_{ik}\bar \ttb_{k+1,k}^{(r+s)},
\eeq
where $\delta_{r,ev}=1$ if $r$ is even and $\delta_{r,ev}=0$ otherwise.

The relation \eqref{bbN2} implies that for $|i-k|>1$, we have
\beq\label{pfN4}
[\bar\ttb_{i+1,i}^{(r)},\bar\ttb_{k+1,k}^{(s)}]=0.
\eeq
Due to \eqref{bibi}, we have
\beq\label{pfN5}
[\bar\ttb_{i+1,i}^{(r)},\bar\ttb_{i+1,i}^{(s)}]=(-1)^r \bar h_{i}^{(r+s)}=(-1)^r(\bar\ttb_{i+1,i+1}^{(r+s)}-\bar\ttb_{ii}^{(r+s)}).
\eeq

If $|i-k|=1$, then using \eqref{bbN}, \eqref{serreN}, and \eqref{pfN3}, we obtain that
\beq\label{pfN6}
[\bar\ttb_{i+1,i}^{(r+1)},\bar\ttb_{k+1,k}^{(s)}]=[\bar\ttb_{i+1,i}^{(r)},\bar\ttb_{k+1,k}^{(s+1)}],
\eeq
\beq\label{pfN7}
[\bar\ttb_{i+1,i}^{(r)},[\bar\ttb_{i+1,i}^{(s)},\bar\ttb_{k+1,k}^{(p)}]]=-[\bar\ttb_{i+1,i}^{(s)},[\bar\ttb_{i+1,i}^{(r)},\bar\ttb_{k+1,k}^{(p)}]]-2(-1)^r\delta_{r+s,ev}\bar\ttb_{k+1,k}^{(r+s+p)}.
\eeq

If $i-j>1$, then we prove that
\beq\label{pfN8}
\bar\ttb_{ij}^{(r)}=[\bar\ttb_{i,i-1}^{(0)},\bar\ttb_{i-1,j}^{(r)}]=[\bar\ttb_{i,j+1}^{(r)},\bar\ttb_{j+1,j}^{(0)}],
\eeq
where the first equality holds by definition \eqref{pfN0}. Indeed, for the second equality, we use induction on $i-j$ with the following,
\begin{align*}
\bar\ttb_{i+1,j}^{(r)}&\stackrel{\eqref{pfN0}}{=}[\bar\ttb_{i+1,i}^{(0)},\bar\ttb_{ij}^{(r)}]=[\bar\ttb_{i+1,i}^{(0)},[\bar\ttb_{i,j+1}^{(r)},\bar\ttb_{j+1,j}^{(0)}]]\\&\stackrel{\eqref{pfN4}}{=}[[\bar\ttb_{i+1,i}^{(0)},\bar\ttb_{i,j+1}^{(r)}],\bar\ttb_{j+1,j}^{(0)}]\stackrel{\eqref{pfN0}}{=}[\bar\ttb_{i+1,j+1}^{(r)},\bar\ttb_{j+1,j}^{(0)}].
\end{align*}
Repeating the same argument with the help of \eqref{pfN8}, we obtain that for $i-j>1$ and $r,s\in\bN$, we have
\beq\label{1pfn}
\bar\ttb_{i+1,j}^{(r+s)}=[\bar\ttb_{i+1,i}^{(r)},\bar\ttb_{i,j}^{(s)}].
\eeq

Now we are in the position of proving \eqref{todo}. We first deal with the case when $i=j$ or $k=l$. Assume that $i=j$. If $k=l$, then the left-hand side of \eqref{todo} is zero by \eqref{hhN} while the right-hand side of \eqref{todo} is clearly zero for either $i\ne k$ or $i=k$. If $k>l$, then we proceed by induction on $k-l$.  It follows from \eqref{pfN1} and \eqref{pfN3} that
\beq\label{pfN2}
[\bar\ttb_{ii}^{(r)},\bar\ttb_{l+1,l}^{(s)}]=2(1-\delta_{r,ev})(\delta_{i,l+1}-\delta_{il})\bar\ttb_{l+1,l}^{(r+s)},
\eeq
confirming \eqref{todo} for this particular case when $k-l=1$. Then the general case follows from
\begin{align*}
[\bar \ttb_{ii}^{(r)},\bar\ttb_{k+1,l}^{(s)}]&=\big[\bar \ttb_{ii}^{(r)},[\bar\ttb_{k+1,k}^{(0)},\bar\ttb_{kl}^{(s)}]\big]\\
&=\big[[\bar \ttb_{ii}^{(r)},\bar\ttb_{k+1,k}^{(0)}],\bar\ttb_{kl}^{(s)}\big]+\big[\bar\ttb_{k+1,k}^{(0)},[\bar \ttb_{ii}^{(r)},\bar\ttb_{kl}^{(s)}]\big]
\end{align*}
using induction hypothesis, \eqref{pfN0}, and \eqref{1pfn}--\eqref{pfN2}. 

Then we shall verify \eqref{todo} case by case, considering all relations between pairs of indices $i>j$ and  $k>l$. There will be seven cases.

(1) If $j>k$, then \eqref{todo} is straightforward by \eqref{pfN4} and \eqref{pfN8}.

(2) If $j=k$, then using \eqref{pfN6} and \eqref{pfN8} we find that
\[
[\bar\ttb_{j+1,j}^{(r)},\bar\ttb_{j,j-1}^{(s)}]=\bar \ttb_{j+1,j-1}^{(r+s)}.
\]
Taking the commutator with $\bar\ttb_{j+2,j+1}^{(0)},\cdots,\bar\ttb_{i+1,i}^{(0)}$ and using \eqref{pfN4}, we obtain
\[
[\bar\ttb_{i+1,j}^{(r)},\bar\ttb_{j,j-1}^{(s)}]=\bar\ttb_{i+1,j-1}^{(r+s)}
\]
Similarly, apply the commutator with $\bar\ttb_{j-1,j-2}^{(0)},\cdots,\bar\ttb_{l+1,l}^{(0)}$ to obtain \eqref{todo} for $j=k$.

(3) Now we prove \eqref{todo} with $i=k$ and $j>l$. We first consider the commutator $[\bar\ttb_{i+3,i+1}^{(r)},\bar\ttb_{i+2,i}^{(s)}]$. It suffices to do the calculation for the case $i=1$ as the same argument works for all $i$. Clearly, it follows from \eqref{pfN4} and \eqref{pfN7} that
\beq\label{pfN9}
[\bar\ttb_{43}^{(r)},[\bar\ttb_{32}^{(0)},[\bar\ttb_{32}^{(0)},\bar\ttb_{21}^{(s)}]]]=[[[\bar\ttb_{43}^{(r)},\bar\ttb_{32}^{(0)}],\bar\ttb_{32}^{(0)}],\bar\ttb_{21}^{(s)}]=0.
\eeq
We have
\begin{align*}
 [\bar\ttb_{42}^{(r)},\bar\ttb_{31}^{(s)}] &\stackrel{\eqref{pfN8}}{=} [[\bar\ttb_{43}^{(r)},\bar\ttb_{32}^{(0)}],[\bar\ttb_{32}^{(0)},\bar\ttb_{21}^{(s)}]]\stackrel{\eqref{pfN9}}{=} [[\bar\ttb_{43}^{(r)},[\bar\ttb_{32}^{(0)},\bar\ttb_{21}^{(s)}]],\bar\ttb_{32}^{(0)}]\\
&\stackrel{\eqref{pfN4}}{=}  [[[\bar\ttb_{43}^{(r)},\bar\ttb_{32}^{(0)}],\bar\ttb_{21}^{(s)}],\bar\ttb_{32}^{(0)}]\stackrel{\eqref{pfN9}}{=} [[\bar\ttb_{43}^{(r)},\bar\ttb_{32}^{(0)}],[\bar\ttb_{21}^{(s)},\bar\ttb_{32}^{(0)}]]\stackrel{\eqref{pfN8}}{=}  - [\bar\ttb_{42}^{(r)},\bar\ttb_{31}^{(s)}].
\end{align*}
Therefore,
\beq\label{pfN10}
[\bar\ttb_{i+3,i+1}^{(r)},\bar\ttb_{i+2,i}^{(s)}]=[\bar\ttb_{42}^{(r)},\bar\ttb_{31}^{(s)}]=0.
\eeq

Then we look at a particular case,
\begin{align*}
[\bar\ttb_{32}^{(r)},\bar\ttb_{31}^{(s)}]&\stackrel{\eqref{pfN0}}{=}[\bar\ttb_{32}^{(r)},[\bar\ttb_{32}^{(0)},\bar\ttb_{21}^{(s)}]] \stackrel{\eqref{pfN7}}{=}-[\bar\ttb_{32}^{(0)},[\bar\ttb_{32}^{(r)},\bar\ttb_{21}^{(s)}]]-2(-1)^r\delta_{r,ev}\bar\ttb_{21}^{(r+s)}\\
&\stackrel{\eqref{pfN6}}{=}-[\bar\ttb_{32}^{(0)},[\bar\ttb_{32}^{(0)},\bar\ttb_{21}^{(r+s)}]]-2(-1)^r\delta_{r,ev}\bar\ttb_{21}^{(r+s)}
\\
&\stackrel{\eqref{pfN7}}{=}\bar\ttb_{21}^{(r+s)}-2(-1)^r\delta_{r,ev}\bar\ttb_{21}^{(r+s)}=(-1)^{r+1}\bar\ttb_{21}^{(r+s)}.
\end{align*}
This implies further that
\begin{align*}
[\bar\ttb_{32}^{(r)},\bar\ttb_{41}^{(s)}]&\stackrel{\eqref{pfN0}}{=}[\bar\ttb_{32}^{(r)},[\bar\ttb_{43}^{(0)},\bar\ttb_{31}^{(s)}]]=[[\bar\ttb_{32}^{(r)},\bar\ttb_{43}^{(0)}],\bar\ttb_{31}^{(s)}]+[[\bar\ttb_{32}^{(r)},\bar\ttb_{31}^{(s)}],\bar\ttb_{43}^{(0)}]\\
& \stackrel{\eqref{pfN0}}{=}-[\bar\ttb_{42}^{(r)},\bar\ttb_{31}^{(s)}]+[(-1)^{r+1}\bar\ttb_{21}^{(r+s)},\bar\ttb_{43}^{(0)}]\xlongequal{\eqref{pfN4}\eqref{pfN10}}0-0=0.
\end{align*} 

Similarly, we have $[\bar\ttb_{l+2,l-1}^{(r)},\bar\ttb_{l+1,l}^{(s)}]=0$. Applying \eqref{pfN8} to it, we get
\beq\label{pfN11}
[\bar\ttb_{ij}^{(r)},\bar\ttb_{l+1,l}^{(s)}]=0
\eeq
for any $i>l+1$ and $j<l$. 

Finally, to show \eqref{todo} for the case $i=k$ and $j>l$, namely 
\beq\label{pfN12}
[\bar\ttb_{ij}^{(r)},\bar\ttb_{il}^{(s)}]=(-1)^{r+1}\bar\ttb_{jl}^{(r+s)},
\eeq
it essentially reduces to show the case $[\bar\ttb_{i,i-1}^{(r)},\bar\ttb_{i,i-2}^{(s)}]$ by using \eqref{pfN8} and \eqref{pfN11}. This case is the same for all $i$, and for $i=3$ it has been done above.

(4) The case $i>k$ and $j=l$ is similar the case (3). Again, the application of \eqref{pfN8} and \eqref{pfN11} reduces the calculation to prove that $[\bar\ttb_{i,i-2}^{(r)},\bar\ttb_{i-1,i-2}^{(s)}]=(-1)^{s+1}\bar\ttb_{i,i-1}^{(r+s)}$.

(5) For the case $i=k$ and $j=l$, we show by induction on $i-j$ that
\[
[\bar\ttb_{ij}^{(r)},\bar\ttb_{ij}^{(s)}]=(-1)^r(\bar\ttb_{ii}^{(r+s)}-\bar\ttb_{jj}^{(r+s)}).
\]
The base case $i-j=1$ is proved in \eqref{pfN5}. By induction hypothesis, we have
\begin{align*}
[\bar\ttb_{ij}^{(r)},\bar\ttb_{ij}^{(s)}]&\stackrel{\eqref{pfN0}}{=}[\bar\ttb_{ij}^{(r)},[\bar\ttb_{i,i-1}^{(0)},\bar\ttb_{i-1,j}^{(s)}]]=[[\bar\ttb_{ij}^{(r)},\bar\ttb_{i,i-1}^{(0)}],\bar\ttb_{i-1,j}^{(s)}]+[\bar\ttb_{i,i-1}^{(0)},[\bar\ttb_{ij}^{(r)},\bar\ttb_{i-1,j}^{(s)}]]\\&
\stackrel{\eqref{pfN12}}{=}[\bar\ttb_{i-1,j}^{(r)},\bar\ttb_{i-1,j}^{(s)}]+[\bar\ttb_{i,i-1}^{(0)},(-1)^{(s+1)}\bar\ttb_{i,i-1}^{(s)}]\\
&=(-1)^{r}(\bar\ttb_{i-1,i-1}^{(r+s)}-\bar\ttb_{jj}^{(r+s)})+(-1)^{s+1}(\bar\ttb_{ii}^{(r+s)}-\bar\ttb_{i-1,i-1}^{(r+s)})=(-1)^{r}(\bar\ttb_{ii}^{(r+s)}-\bar\ttb_{jj}^{(r+s)}).
\end{align*}

(6) If $i>k>j>l$, then due to \eqref{pfN8} and \eqref{pfN11}, it suffices to verify that $[\bar\ttb_{i+3,i+1}^{(r)},\bar\ttb_{i+2,i}^{(s)}]=0$ which is done in \eqref{pfN10}.

(7) If $i>k>l>j$, then similarly it is reduced to \eqref{pfN11}.

Thus, we complete the verification of the relations \eqref{todo}. They imply that the associated graded algebra $\gr\,\YDrN$ is spanned by the set of monomials in the elements $\ttb_{ij}^{(r)}$, $1\lle j<i\lle N$, and $\ttb_{ii}^{(2r+1)}$, $1\lle i\lle N$, $r\in\bN$ taken in some fixed order. We shall take the order that all $\ttb_{ii}^{(2r+1)}$ are to the left of $\ttb_{jk}^{(s)}$. Due to \eqref{pfN1} and \eqref{hhN}, we have the associated graded algebra $\gr\,\YDrN$ is spanned by the set of monomials in the elements $\ttb_{ij}^{(r)}$, $1\lle j<i\lle N$, and $h_{i,2r+1}$, $0\lle i< N$, $r\in\bN$ taken in some fixed order that all $h_{i,2r+1}$ are to the left of $\ttb_{jk}^{(s)}$, proving of injectivity of the homomorphism \eqref{surjm}. This completes the proof of the theorem.
\end{proof}

\end{document}